%% file: LpMahlerv2.tex
\documentclass[11pt]{article}
\usepackage[utf8]{inputenc}
\input{macros}

\usepackage{array}
\newcolumntype{L}[1]{>{\raggedright\let\newline\\\arraybackslash\hspace{0pt}}m{#1}}
\newcolumntype{C}[1]{>{\centering\let\newline\\\arraybackslash\hspace{0pt}}m{#1}}
\newcolumntype{R}[1]{>{\raggedleft\let\newline\\\arraybackslash\hspace{0pt}}m{#1}}

\numberwithin{equation}{section}

\title{$L^p$-Legendre transforms and Mahler integrals: \\
Asymptotics and the Fokker–Planck heat flow}
\author{Vlassis Mastrantonis}
\date{}

\begin{document}

\maketitle

\begin{abstract}
$L^p$-polarity and $L^p$-Mahler volumes were recently introduced by Berndtsson, Rubinstein, and the author as a new approach, inspired by complex geometry, to the Mahler, Bourgain, and B\l{}ocki Conjectures.
This paper serves two purposes. First, it introduces functional analogues of these notions and establishes functional versions of key theorems previously formulated in the setting of convex bodies.
This involves introducing the $L^p$-Legendre transform and analyzing the associated Santal\'o points and the dimensional asymptotics of the Mahler volumes of the conjectured extremizers. 
Second, the paper investigates the connection between the $L^p$-Legendre transform and the recent work of Nakamura--Tsuji on the Fokker--Planck heat flow. As a byproduct, a functional $L^p$ Santal\'o inequality is established. 
The proof is based on deriving the evolution equations for the $L^p$-Legendre transform and Mahler integral under the Fokker--Planck heat flow. 
A second approach using the geometric method of Artstein--Klartag--Milman is also presented, for which the necessary asymptotics are derived in the $L^1$-case. 

\end{abstract}

\bigskip 
\noindent

\section{Introduction}
In a recent program initiated jointly with Rubinstein \cite{MR22, Mastr23, MR24, Rubi24} and Berndtsson--Rubinstein \cite{BMR23}, we introduced a notion of $L^p$-polarity for convex bodies and defined the corresponding $L^p$-Mahler volume, generalizing the classical notion of polarity, support function, and Mahler volume. This work extended the classical notion of the Santal\'o point to the $L^p$-setting \cite[Proposition 1.5]{BMR23}, and established an $L^p$ Santal\'o inequality \cite[Theorem 1.6]{BMR23} \cite[Theorem 1.2]{Mastr23}. 
Moreover, the $L^p$-Santal\'o point turned out to be crucial for a complex geometric approach to Bourgain's conjecture \cite[Theorem 1.9]{BMR23}.
Additionally, we proposed $L^p$-versions of Mahler's conjectures \cite[Conjectures 1.3 and 1.4]{BMR23}, which we resolved in dimension two \cite[Theorem 1.9 and 1.10]{MR24}. The primary aim of this paper is to extend these concepts to functions by introducing an $L^p$-Legendre transform and studying the corresponding $L^p$-Mahler integral. A secondary aim is to point out relations of our work to recent work of Nakamura--Tsuji and an old remark by Tao.

The main result is establishing a functional $L^p$-Santal\'o inequality (Theorem \ref{LpFuncSantalo}). We offer two approaches. 

The first approach, inspired by the work of Artstein--Klartag--Milman, involves overcoming a significant technical obstacle. Namely, the precise dimensional asymptotic estimate for the $L^p$-Mahler volume of the Euclidean ball. We address this in the most interesting case of $p=1$, which corresponds to the Bergman kernels of tube domains \cite[(1--4)]{BMR23}, by explicitly computing the asymptotics of the $L^1$-Mahler volume of the ball. We leave the approximation of the Mahler integral by the Mahler volume of higher-dimensional bodies to a subsequent paper \cite{Mastr24}.

The second approach is inspired by the method developed by Nakamura--Tsuji using the Fokker--Planck heat flow \cite{NakamuraTsuji22, NakamuraTsuji24, CGNT24}. Recently, they established a sharp inequality on the $L^p$-norms of the Laplace transform for log-concave functions \cite[Corollary 1.6]{NakamuraTsuji24}, which directly corresponds to the functional $L^p$-Santal\'o inequality for even convex functions, as proposed here. In the following sections, we explore the relationship between their inequality and Theorem \ref{LpFuncSantalo}. The key is computing the evolution equations for the $L^p$-Legendre transform and $L^p$-Mahler integral under the Fokker--Planck heat flow, demonstrating the monotonicity of the Mahler integral after time-dependent translation by the $L^p$-Santal\'o point. This implies that the limiting function $|x|^2/2$ serves as the maximizer.

\subsection{\texorpdfstring{$L^p$-polarity}{Lp-polarity}}
The following generalizes the classical support function \cite[(1-8)]{BMR23}: 
\begin{definition}
\label{LpSupportDef}
Let $p\in (0,\infty)$. For a convex body $K\subset \R^n$, let
\begin{equation*}
    h_{p,K}(y)\defeq \frac1p \log\int_{K} e^{p\langle x,y\rangle}\frac{dx}{|K|} 
\end{equation*}    
be the $L^p$-support function of $K$. 
\end{definition}

It follows directly from H\"older's inequality that $h_{p,K}$ is a monotone increasing family of functions \cite[Lemma 2.4(v)]{BMR23}. The $p=\infty$ case corresponds to the classical support function, $h_{\infty, K}(y)\defeq \lim_{p\to\infty} h_{p,K}(y)= \sup_{x\in\R^n}\langle x,y\rangle= h_K(y)$, as can be verified by evaluating the limit \cite[Corollary 2.7]{BMR23}. However, for finite $p$, $h_{p,K}$ loses homogeneity, a hallmark property of the classical support function. Consequently, defining an $L^p$-polar body requires constructing a norm from $h_{p,K}$. This is done via Ball's classical construction \cite[Theorem 5]{Ball88}. Let
\begin{equation*}
    K^{\circ,p}\defeq \left\{y\in\R^n: \|y\|_{K^{\circ,p}}\defeq \left( \frac{1}{(n-1)!}\int_0^\infty r^{n-1} e^{-h_{p,K}(ry)}dr\right)^{-\frac1n}\leq 1\right\}
\end{equation*}
be the $L^p$-polar body of $K$ \cite[Definition 1.1]{BMR23}. This is such that $n! |K^{\circ,p}|= \int e^{-h_{p,K}}$ \cite[Theorem 1.2]{BMR23}. Note that homogeneity of $h_K$ implies 
$$K^{\circ, \infty}= \{y\in \R^n: h_K(y)\leq 1\}= K^\circ,$$ 
recovering the classical polar body. 
The main distinction between $L^p$-polarity and classical polarity is the loss of duality: in general, $(K^{\circ,p})^{\circ,p} \neq K$. This can be seen most clearly when $K$ is a polytope, as $(K^{\circ,p})^{\circ,p}$ has a smooth boundary, whereas $K^\circ$ does not.

Motivated by the well-known formula $\M(K)=|K|\int e^{-h_K}$, define the $L^p$-Mahler volume of $K$ via
\begin{equation*}
    \M_p(K)\defeq |K|\int_{\R^n} e^{-h_{p,K}(y)}dy = n! |K| |K^{\circ,p}| 
\end{equation*}
\cite[(1-2)]{BMR23}. 
This definition satisfies $\lim_{p\to\infty}\M_p(K)=\M(K)\defeq n!|K||K^\circ|$. One can formulate conjectures analogous to Mahler's \cite[Conjectures 1.3--1.4]{BMR23}: 
Let
\begin{equation*}
    \Delta_{n,0}\defeq \mathrm{conv}\{e_1, \ldots, e_n, -(e_1 + \ldots + e_n)\}
\end{equation*}
denote a centered simplex, where $e_1, \ldots, e_n$ are the standard basis vectors of $\R^n$. 
\begin{conjecture}
\label{LpMahlerConjSym}
    Let $p\in (0,\infty]$. For a symmetric convex body $K\subset \R^n, \M_p(K)\geq \M_p([-1,1]^n)$. 
\end{conjecture}
\begin{conjecture}
\label{LpMahlerConj}
    Let $p\in (0,\infty]$. For a convex body $K\subset \R^n$, $\M_p(K)\geq \M_p(\Delta_{n,0})$. 
\end{conjecture}
The $p=\infty$ case corresponds to Mahler's conjectures \cite{Mahler39b, Mahler39a, Mahler50}.
A potential advantage of Conjecture \ref{LpMahlerConjSym} for finite $p$ is that $[-1,1]^n$ appears to be the unique minimizer (modulo linear transformations). This assertion is currently only supported by numerical evidence \cite[\S 3E]{BMR23} \cite[(13)]{Blocki15}. Notably, we have verified both Conjectures \ref{LpMahlerConjSym} and \ref{LpMahlerConj} in dimension two \cite{MR24}. Additionally, a Santal\'o inequality holds for $\M_p$: Let
\begin{equation*}
    B_2^n\defeq \{x\in \R^n: |x|\defeq \sqrt{x_1^2 + \ldots+ x_n^2}\leq 1\}
\end{equation*}
denote the Euclidean unit ball. For symmetric convex bodies, $\M_p(K)\leq \M_p(B_2^n)$  \cite[Theorem 1.6]{BMR23}. In general, the body needs to be translated by a ``Santal\'o point". The notion of a Santal\'o point \cite[155]{Santalo49} extends to the $L^p$-Mahler volume \cite[Proposition 1.5]{BMR23}. 
\begin{proposition}
\label{LpSantaloPoint}
    Let $p\in (0,\infty)$. For a convex body $K\subset \R^n$ there exists a unique $s_p(K)\in \R^n$ such that 
    \begin{equation*}
        \M_p(K-s_p(K))= \inf_{x\in \R^n}\M_p(K), 
    \end{equation*}
    uniquely characterized by the vanishing of the barycenter of $h_{p, K-s_p(K)}$. Moreover, $s_p(K)\in\mathrm{int}\,K$. 
\end{proposition}
We call $s_p(\phi)$ the \textit{$L^p$-Santal\'o point} of $\phi$. 

Santal\'o's inequality \cite{Santalo49} extends to the $L^p$-Mahler volume \cite[Theorem 1.2]{Mastr23}
\begin{theorem}
\label{LpSantaloThm}
    Let $p\in (0,\infty)$. For a convex body $K\subset \R^n$, 
    $
        \inf_{x\in\R^n}\M_p(K-x)\leq \M_p(B_2^n). 
    $
\end{theorem}

In this paper, we aim to address the following question: 
\begin{question}
\label{GoalQuestion}
    Can analogous notions be defined for convex functions? 
\end{question}

\subsection{Functional analogues}
Throughout this paper, we consider functions $f:\R^n\to \R\cup\{\infty\}$ that are not $\infty$ almost everywhere. Equivalently, the \textit{volume} of $f$: 
\begin{equation*}
    V(f)\defeq \int_{\R^n} e^{-f(x)}dx 
\end{equation*}
is positive.  
For example, 
\begin{equation}
\label{|x|^2/2Vol}
       V(|x|^2/2) = \int_{\R^n}e^{-|x|^2/2}dx = \left(\int_{\R^2}e^{-(x^2+y^2)/2} dx dy\right)^{n/2} \!\!\!= \left( 2\pi \int_{\R}e^{-r^2/2}rdr\right)^{n/2} = (2\pi)^{n/2}. 
\end{equation}
The definition generalizes the volume of a convex body $K\subset \R^n$ as follows: the volume of the convex indicator function
\begin{equation*}
    \bm{1}_K^\infty(x)\defeq \begin{dcases}
        0, & x\in K, \\
        \infty, & x\not\in K, 
    \end{dcases}
\end{equation*}
is exactly the volume of the underlying convex body: $V(\bm{1}_K^\infty)= \int_{\R^n}e^{-\bm{1}^\infty_K(x)}dx= \int_K dx = |K|$. 

The Mahler problem is traditionally studied not for all convex sets, but rather for those that have non-empty interiors and are compact. This restriction avoids cases where the Mahler volume may be ill-defined, such as when the product of $0$ and $\infty$ arises. Similarly, to formulate the Mahler problem for convex functions, we must impose additional conditions: Convexity of the set is replaced by convexity of the function. Non-emptiness of the interior, by positivity of the volume. Boundedness, for convex bodies with non-empty interior, is equivalent to finiteness of volume (Lemma \ref{FiniteVolClaim}), hence we adopt this condition for functions. Finally, closedness corresponds to lower semi-continuity. Recall that a function $f:\R^n\to \R\cup\{\infty\}$ is lower semi-continuous if 
\begin{equation}
\label{LSC}
    f(x_0)\leq \liminf_{x\to x_0}f(x), \quad \text{for all } x_0\in \R^n. 
\end{equation}
This condition ensures that the sublevel sets $\{f\leq M\}$ are closed. To summarize: 
we restrict to the class
\begin{equation*}
    \Cvx(\R^n)\defeq \{\phi: \R^n\to \R\cup\{\infty\} \text{ proper, convex, and lower semi-continuous with } 0\!<\!V(\phi)\!<\!\infty\}.
\end{equation*}
Notably, any convex function can be slightly altered to a lower semi-continuous function. In particular, the \textit{lower semi-continuous hull} of $f$ \cite[p. 52]{Rock70}
$
    (\mathrm{cl}\,f)(x)\defeq \liminf_{y\to x}f(y)
$
is the greatest lower semi-continuous function majorized by $f$. For a convex function $\phi$, $\mathrm{cl}\,\phi$ is itself convex, and differs from $\phi$ only on a set of zero Lebsgue measure \cite[Theorem 7.4]{Rock70}.

Let us now begin addressing Question \ref{GoalQuestion}. 
To extend Definition \ref{LpSupportDef} to the functional setting, we introduce the following: 
\begin{definition}
\label{LpLegendreDef}
    Let $p\in (0, \infty)$. For $f:\R^n\to \R\cup\{\infty\}$ with $0<V(f)<\infty$, let 
    \begin{equation*}
        f^{*,p}(y)\defeq \frac1p\log\int_{\R^n} e^{p(\langle x,y\rangle - f(x))} \frac{e^{-f(x)}dx}{V(f)}. 
    \end{equation*}
    be the $L^p$-Legendre transform of $f$. 
\end{definition}
\noindent 
With this definition, for the convex indicator function of a convex body $K$, 
\begin{equation}
\label{LpLegendreK}
    (\bm{1}_K^\infty)^{*,p}(y) = \frac1p \log \int_{K}e^{p\langle x,y\rangle}\frac{dx}{|K|} = h_{p,K}(y), 
\end{equation} 
recovers the $L^p$-support function. However, to recover the $L^p$-support function, any scaling of $f$ would work in Definition \ref{LpLegendreDef}. The reason why $p+1$ is the appropriate scaling, is that this choice aligns the $L^p$-Legendre transform with the classical Legendre transform in the same way that the $L^p$-support function corresponds to the classical support function. Namely, $f^{*,p}$ should be increasing in $p$ and dominated by $f^*$ with $f^{*,\infty}\defeq \lim_{p\to\infty} f^{*,p}= f^*$. We demonstrate this property for convex functions in Corollary \ref{LpConvergence}. 

For the corresponding Mahler integral, let 
    \begin{equation*}
        \M_p(f)\defeq V(f)V(f^{*,p})= \int_{\R^n} e^{-f(x)}dx \int_{\R^n} e^{-f^{*,p}(y)}dy
    \end{equation*}
    be the \textit{$L^p$-Mahler integral} of $f$. 
By \eqref{LpLegendreK}, $\M_p(\bm{1}_K^\infty)= \M_p(K)$. 
We propose the following conjectures: 
\begin{conjecture}
\label{SymLpFunMahlerConj}
    Let $p\in (0,\infty]$. For an even $\phi\in \Cvx(\R^n)$, 
    $
        \M_p(\phi)\geq \M_p([-1,1]^n).
    $
\end{conjecture}
\begin{conjecture}
\label{LpFuncMahlerConj}
    Let $p\in (0,\infty]$. For $\phi\in \Cvx(\R^n)$, 
    $
        \M_p(\phi)\geq \Big( e^{1+\frac1p} p^{\frac1p}\Gamma(1+\frac1p)\Big)^{n}.
    $
\end{conjecture}
The $p=\infty$ case corresponds to the functional Mahler conjectures as posed by Fradelizi--Meyer \cite[Conjectures (1'), (2')]{FM08a}.  

In extending Proposition \ref{LpFuncSantalo} for functions, along with the work of Artstein--Klartag--Milman on the Santal\'o point of a function \cite{AKM04}, we require the following: The functional analogue of translation is defined by
\begin{equation*}
    (T_af)(x)\defeq f(x+a),
\end{equation*}
since for the convex indicator function, $T_a\bm{1}_K^\infty= \bm{1}_{K-a}^\infty$. 
Additionally, for a measurable $f$ with $0<V(f)<\infty$, let 
\begin{equation*}
    b(f)\defeq \int_{\R^n}x\frac{e^{-f(x)}dx}{V(f)}
\end{equation*}
be the \textit{barycenter} of $f$. We demonstrate the following: 

\begin{proposition}
\label{LpFunSantaloPoint}
    Let $p\in (0,\infty)$. For $\phi\in Cvx(\R^n)$, there exists unique $s_p(\phi)\in \R^n$ with
    \begin{equation*}
        \M_p(T_{s_p(\phi)}\phi)= \inf_{x\in \R^n}\M_p(T_x\phi), 
    \end{equation*}
    which is also the unique point such that $b((T_{s_p(\phi)}\phi)^{*,p})=0$. Moreover, $s_p(\phi)\in \mathrm{int}\,\{\phi<\infty\}$. 
\end{proposition}

The proof of Proposition \ref{LpFunSantaloPoint} occupies \S \ref{LpSantaloPointSection}. 
For the upper bound we aim to prove the following: 
\begin{theorem}
\label{LpFuncSantalo}
    Let $p\in(0,\infty)$. For $\phi\in\Cvx(\R^n)$, $\inf_{x\in\R^n}\M_p(T_{x}\phi)\leq \M_p(|x|^2/2)$. 
\end{theorem}

Convexity is not directly used in the proof of Theorem \ref{LpFuncSantalo}. Instead, it is leveraged for its imposed growth condition via Lemma \ref{ConvexLowerBoundLemma}. Since Lemma \ref{ConvexLowerBoundLemma} is not true in general without the assumption of convexity, Theorem \ref{LpFuncSantalo} can be extended to non-convex functions by assuming superlinearity. It should not come as a surprise that Santal\'o-type theorems do not require convexity: 
\begin{corollary}
\label{GeneralLpFunSantalo}
    Let $f: \R^n\to \R\cup \{\infty\}$ be a measurable function such that there exist $a>0$ and $b\in \R$ with $f(x)\geq a|x| + b$. Then, $\inf_{x\in\R^n}\M_p(T_xf)\leq \M_p(|x|^2/2)$. 
\end{corollary}

Our attempt to prove Theorem \ref{LpFuncSantalo} via the method of Artstein--Klartag--Milman \cite{AKM04}, by approximating the Mahler integral of a function via the Mahler volume of convex bodies and applying the inequality for bodies, is yet incomplete. Nonetheless, we discuss our progress towards this direction in \S\ref{Asymptotics}. The main ingredient is the calculation of the asymptotics in dimension for $\M_p(B_2^n)$. For finite $p$, this requires evaluating an integral involving negative powers of the modified Bessel function of the first kind, making this case significantly more complex than the $p=\infty$ case. However, Theorem \ref{LpFuncSantalo} may be obtained using Nakamura--Tsuji's approach via the Fokker--Planck heat flow with minimal modifications. We lay out the details in \S \ref{FokkerPlanckSection}.

\subsection{Relation to a comment of Klartag and Tao}
One of the motivations for Definition \ref{LpLegendreDef} is that it ties neatly to a comment by Tao in his blog post in response to a suggestion by Klartag \cite{TaoBlog}. Tao suggested the following: For an exponentially decaying log-concave function $F:\R^n\to (0,\infty)$, denote its Laplace transform via
\begin{equation*}
    \mathcal{L}F (y)\defeq \int_{\R^n} F(x)e^{\langle x,y\rangle}dx. 
\end{equation*}
If one could establish a Hausdorff--Young type inequality
\begin{equation}
\label{TaoConjecture}
    \frac{\|F\|_{L^a(\R^n)}}{\|\mathcal{L}F\|_{L^{\frac{a}{a-1}}(\R^n)}}\geq (e+ g(a, n))^{\frac{n}{a}}
\end{equation}
for some $g$ with $\lim_{a\to 0}g(a)=0$, then Mahler's conjecture would follow by taking $a\to 0$. The following Lemma shows how Tao's approach is related to Conjectures \ref{SymLpFunMahlerConj} and \ref{LpFuncMahlerConj}. 

\begin{lemma}
\label{DetropicalizedEquiv}
    Let $a\in (0,1)$ and $p\defeq \frac{1-a}{a}\in(0,\infty)$. For any non-negative $F\in L^a(\R^n)$, let $\phi\defeq -a\log F$. Then, 
    \begin{equation*}
        \frac{\|F\|_{L^a}}{\|\mathcal{L}F\|_{L^{\frac{a}{a-1}}}} = p^{np}\M_p(\phi)^p. 
    \end{equation*}
\end{lemma}
We prove Lemma \ref{DetropicalizedEquiv} in \S \ref{NakamuraTsujiConnection}. For now, let us demonstrate how, by Lemma \ref{DetropicalizedEquiv}, Conjecture \ref{LpFuncMahlerConj} implies Tao's comment \eqref{TaoConjecture}: Suppose Conjecture \ref{LpFuncMahlerConj} holds. Then, by Lemma \ref{DetropicalizedEquiv} and the fact that $p=\frac{1-a}{a}$, 
\begin{equation*}
    \frac{\|F\|_{L^a}}{\|\mathcal{L}F\|_{L^{\frac{a}{a-1}}}}\geq p^{np}(e^{1+\frac1p}p^{\frac1p}\Gamma(1+\frac1p))^{np}= \left(e - e +\frac{e}{1-a}\Gamma(\frac{1}{1-a})^{1-a} \right)^{\frac{n}{a}}, 
\end{equation*}
where $\lim_{a\to 0}( -e + \frac{e}{1-a}\Gamma(\frac{1}{1-a})^{1-a})= -e + e =0$. In sum:
\begin{proposition}
    If Conjecture \ref{LpFuncMahlerConj} holds for some $p\in (0,\infty)$, then \eqref{TaoConjecture} holds with $a= \frac{1}{p+1}$. 
\end{proposition}

\subsection{The Fokker--Planck heat flow}
Nakamura--Tsuji recently initiated a study of the Mahler volume via the Ornstein--Uhlenbeck semigroup of a log-concave function and its dual \cite{NakamuraTsuji22}:  
\begin{equation}
\label{OUsemigroup}
    (P_tf)(t,x)\defeq \int_{\R^n} f(y) e^{-\frac{|x- e^{-t} y|^2}{2(1-e^{-2t})}} \frac{dx}{(2\pi (1-e^{-2t})^{\frac{n}{2}})}dt, \quad (t,x)\in [0,\infty)\times\R^n. 
\end{equation}
The key observation is that $P_tf$ is a solution to the following parabolic equation: 
\begin{equation}
\label{FPEq1}
    \begin{dcases}
        (\partial_t f) (t,x)= (\Delta f)(t,x) + \langle x, Df(t,x)\rangle + nf(t,x) &
        (t,x)\in [0,\infty) \times \R^n, \\
        f(0, x) = f_0(x), &\quad x\in\R^n, 
    \end{dcases}
\end{equation}
known as the Fokker--Planck heat flow, which regardless of the initial condition, converges to a Gaussian. This flow can thus interpolate between any log-concave function (resp. convex function) and a Gaussian (resp. convex quadratic polynomial). Consequently, demonstrating monotonicity properties for the desired functional along this flow allows one to establish Gaussians, or equivalently convex quadratic polynomials, as extremizers. 
Nakamura--Tsuji demonstrated such monotonicity for the quotient that appears on the left-hand side of \eqref{TaoConjecture} for even functions \cite[Theorem 1.2]{NakamuraTsuji24}. As a result, they obtained a sharp upper bound on this quotient, attained by Gaussians \cite[Corollary 1.6]{NakamuraTsuji24}. In view of Lemma \ref{DetropicalizedEquiv}, this result is exactly equivalent to Theorem \ref{LpFuncSantalo} for even functions. 

Is \S \ref{FokkerPlanckSection}, we streamline the proof of Theorem \ref{LpFuncSantalo} using the Fokker--Planck heat flow by explicitly computing the evolution equations of the involved terms. First, by representing log-concave functions as $f= e^{-\phi}$ for a convex function $\phi$, equation \eqref{FPEq1} induces a flow on convex functions that converges to a convex quadratic polynomial. Second, \eqref{OUsemigroup} provides an integral representation for the solution, from which it becomes clear that the volume of the function remains constant under this flow (Lemma \ref{VolEvolution}). As a result, it suffices to study the evolution equation of $V(\phi^{*,p})$. We compute the evolution equation of $\phi^{*,p}$ (Lemma \ref{LpLegendreEvolution}), involving the Fischer information of a measure closely related to the Hessian of $\phi^{*,p}$ (Lemma \ref{D^2CovEq}). Then we compute the evolution equation of $\M_p(\phi)$ (Lemma \ref{LpMahlerIntegralEvolution}).
Using this together with the Cram\'er--Rao (Theorem \ref{CramerRaoIneq}) and a variance inequality (Theorem \ref{VarianceIneq}), we bound $\partial_t\M_p(\phi)$ from below:
\begin{proposition}
\label{Mpbound}
    Let $p\in (0,\infty)$. For $\phi$ such that $e^{-\phi}$ is a solution to \eqref{FPEq1}, 
    \begin{equation*}
        \partial_t \M_p(\phi)\geq -\frac{p}{p+1} \M_p(\phi)\, |b(\phi^{*,p})|^2. 
    \end{equation*}
\end{proposition}
Therefore, as long as $\phi$ is translated to keep the $L^p$-Santal\'o point at the origin (then by Proposition \ref{LpFunSantaloPoint}, $b(\phi^{*,p})=0$), $\M_p(\phi)$ is monotone increasing under the Fokker--Planck heat flow, proving Theorem \ref{LpFuncSantalo}. The details are laid out in \S \ref{FinishingProofSection}.

\bigskip
\noindent
\textbf{Organization}. Section \ref{LegendreTransformSection} consists of two parts: In \S \ref{LpLegendreProperties}, a few fundamental properties of the $L^p$-Legendre transform are presented, including convexity, inequalities between different indices, and tensoriality. \S \ref{LpLegendreExamples} provides explicitly computations for examples including $|x|^2/2$, the Euclidean ball, and the simplex. Similarly, \S\ref{LpMahlerIntegralProperties} outlines basic properties of the $L^p$-Mahler integral, while \S\ref{LpMahlerIntegralExamples} presents several examples. \S \ref{NakamuraTsujiConnection} explores connections to the work of Nakamura--Tsuji. Section \ref{LpSantaloPointSection} is dedicated to the proof of Proposition \ref{LpSantaloPoint}. Section \ref{Asymptotics} discusses the asymptotics of the $L^p$-Mahler volume of the Euclidean ball. Finally, in Section \ref{FokkerPlanckSection} we compute the evolution equations of the $L^p$-Legendre transform and the $L^p$-Mahler integral under the Fokker--Planck heat flow, and demonstrate how they can be used to prove Theorem \ref{LpFuncSantalo}. 

\bigskip
This paper comprises part of the author's Ph.D. thesis and extends our previous work \cite{BMR23, Mastr23} from the case of convex bodies to convex functions. As this paper was nearing completion, the author was informed of a very recently posted preprint \cite{CFL24} that substantially overlaps some of our results. 

\bigskip
\noindent 
\textbf{Acknowledgments.}
The author would like to thank Yanir Rubinstein for his continuous support and guidance, and H. Tsuji for his interest in this work.
Research supported by NSF grants 
DMS-1906370,2204347, BSF 2020329 and an Ann G. Wylie Dissertation Fellowship Award at the University of Maryland.

\section{\texorpdfstring{$L^p$-Legendre Transform}{Lp Legendre Transform}}
\label{LegendreTransformSection}
This section is divided into two subsections. In the first, we establish several fundamental properties of the $L^p$-Legendre transform, focusing on convexity, smoothness, tensoriality, and the relationship between transforms for different values of $p$. In the second subsection, we explicitly compute the $L^p$-Legendre transform of specific functions, including the $L^1$-norm, a convex quadratic polynomial, the functional simplex, the Euclidean ball, and the simplex in $\R^n$. 

\subsection{Basic properties}
\label{LpLegendreProperties}
Let us start by proving a few basic properties of the $L^p$-Legendre transform. While most properties carry over for the case of convex bodies, the $L^p$-Legendre transform lacks certain features of the $L^p$-support function. This difference arises because the geometric case corresponds to convex indicator functions, which satisfy
\begin{equation*}
    (p+1) \bm{1}^\infty_K  = \bm{1}_K^\infty, 
\end{equation*}
for all $p>0$. However, this property does not hold when $\bm{1}_K^\infty$ is replaced by a general convex function. Let us list the main differences: 

\begin{itemize}
    \item One difference is that for convex bodies, $h_{p,K}(y)$ is finite for all $y\in \R^n$ (this follows from the finiteness of the classical support function $h_K$). While this does not generally hold for convex functions, we do know that for $\phi\in\Cvx(\R^n)$, the set $\{\phi^{*,p}<\infty\}$ has non-empty interior (Corollary \ref{LpLegendreNonEmptyInterior}). 
    
    \item Additionally, for convex bodies, knowing $h_{1,K}$ alone suffices to fully recover the convex body, as $h_{p,K}(y)= \frac1p h_{1,K}(py)$. Taking $p\to \infty$ recovers the support function $h_K$ \cite[Corollary 2.7]{BMR23}, hence the convex body. However, this relationship does not extend to the functional case. 
    
    \item Finally, the comparison between $\phi^{*,p}$ and $\phi^{*,q}$ (Lemma \ref{ReverseIneq}) is more rigid compared to the case of convex bodies, where it is possible to relate the transforms between different values of $p$ and $q$ more flexibly \cite[Lemma 2.6]{BMR23}. 
    
\end{itemize}

\subsubsection{Finiteness properties}
At this stage, we are not concerned with whether the transform is finite at specific values, as many of the forthcoming results remain valid even at points where the involved functions take the value $\infty$. However, it is key that $-\infty$ is never attained: 
\begin{lemma}
Let $p\in (0,\infty)$. 
For $f:\R^n\to \R\cup\{\infty\}$ with $0<V(f)<\infty$, $f^{*,p}>-\infty$. 
\end{lemma}
\begin{proof}
The condition $V(f)>0$ implies that
\begin{equation*}
    \{f<\infty\} = \bigcup_{m\geq 1}\{f\leq m\}
\end{equation*}
has positive Lebesgue measure. Therefore, there exists some $m_0$ such that $\{f\leq m_0\}$ has positive Lebesgue measure. As a result, 
\begin{equation}
\label{negInfEq1}
    f^{*,p}(y)\geq \frac1p\log\int_{\{f\leq m_0\}}e^{p\langle x,y\rangle - (p+1)f(x)}\frac{dx}{V(f)}\geq \frac1p \log\int_{\{f\leq m_0\}} e^{p\langle x,y\rangle} e^{-(p+1)m_0} \frac{dx}{V(f)}
\end{equation}
Further restrict the domain by choosing large enough $R>0$ such that $\{f\leq m_0\}\cap RB_2^n$ has positive Lebesgue measure. For $x\in RB_2^n$, $\langle x,y\rangle\geq -|x||y|\geq -R|y|$. Therefore \eqref{negInfEq1} becomes:
\begin{equation*}
    f^{*,p}(y)\geq \frac1p \log\int_{\{f\leq m_0\}\cap RB_2^n} e^{-R|y|-(p+1)m_0}\frac{dx}{V(f)} =\frac1p\log\left( \frac{|\{f\leq m_0\}\cap RB_2^n|}{e^{R|y|+(p+1)m_0} V(f)}\right). 
\end{equation*}
Since $|\{f\leq m_0\}\cap RB_2^n|$ was chosen to be positive, the above expression is greater than $-\infty$. 
\end{proof}

As we will later see, 
it is possible for a function $f$ with $0<V(f)<\infty$ to have an $L^p$-Legendre transform of infinite volume, $V(f^{*,p})=\infty$. However, we currently do not know of any example where $V(f^{*,p})=0$. Such a case would imply that $f^{*,p}$ is $\infty$ almost everywhere, which seems unlikely. We can, however, prove that this does not occur for convex functions. 
\begin{lemma}
\label{finiteLemma2}
    Let $p\in (0,\infty)$. For $\phi\in\Cvx(\R^n)$, $V(\phi^{*,p})>0$. 
\end{lemma}
\begin{proof}
By Lemma \ref{ConvexLowerBoundLemma}, there exist $a>0$ and $b\in \R$ such that $\phi(x)\geq a|x|+b$, $x\in \R^n$. Therefore,
\begin{equation*}
    \int_{\R^n} e^{p\langle x,y\rangle - (p+1)\phi(x)}\frac{dx}{V(\phi)}\leq \int_{\R^n}e^{p|y||x| -(p+1)a|x|- (p+1)b}\frac{dx}{V(\phi)}  = \int_{\R^n}e^{p(|y| - \frac{a(p+1)}{p})|x|} \frac{dx}{V(\phi)} e^{-(p+1)}b. 
\end{equation*}
For $|y|<\frac{a(p+1)}{p}$ the integral is finite, hence $\mathrm{int}(\frac{a(p+1)}{p}\,B_2^n) \subset \{\phi^{*,p}<\infty\}$. In particular, $V(e^{-\phi^{*,p}})>0$.
\end{proof}

As we will later see it can be $V(\phi^{*,p})=\infty$. In fact, this is equivalent to $0\notin\mathrm{int}\,\{\phi<\infty\}$ (Lemma \ref{FinitenessLemma}).

\begin{corollary}
\label{LpLegendreNonEmptyInterior}
    Let $p\in (0,\infty)$. For $\phi\in\Cvx(\R^n)$, $\mathrm{int}\,\{\phi^{*,p}<\infty\}\neq\emptyset$.
\end{corollary}
\begin{proof}
    Since, by Lemma \ref{finiteLemma2}, $V(\phi^{*,p})= \int_{\{\phi^{*,p}<\infty\}}e^{-\phi^{*,p}(y)}dy= \int_{\mathrm{int}\,\{\phi^{*,p}<\infty\}}$ is non-zero, it cannot be that $\mathrm{int}\,\{\phi^{*,p}<\infty\}$ is empty. 
\end{proof}

\subsubsection{Convexity}
\begin{lemma}
\label{LpLegendreConvexity}
    Let $p\in (0,\infty)$. For $f:\R^n\to \R\cup\{\infty\}$, $f^{*,p}(y)$ is a convex function of $y$. 
\end{lemma}
\begin{proof}
Let $y_1, y_2\in \R^n$ and $\lambda\in [0,1]$. By H\"older's inequality, 
\begin{equation*}
\begin{aligned}
    &\int_{\R^n} e^{p(\langle x,(1-\lambda)y_1 + \lambda y_2\rangle - f(x))} \frac{e^{-f(x)}dx}{V(f)} \\
    &= \int_{\R^n}\left( e^{p(\langle x, y_1\rangle - f(x))}\frac{e^{-f(x)}}{V(f)}\right)^{1-\lambda} \left( e^{p(\langle x,y_2\rangle - f(x))}\frac{e^{-f(x)}}{V(f)}\right)^\lambda dx \\
    &\leq \left( \int_{\R^n}e^{p(\langle x, y_1\rangle- f(x))} \frac{e^{-f(x)}dx}{V(f)}\right)^{1-\lambda}  
    \left( \int_{\R^n}e^{p(\langle x, y_2\rangle- f(x))} \frac{e^{-f(x)}dx}{V(f)}\right)^{\lambda}. 
\end{aligned}
\end{equation*}
Applying $\log$ to both sides and dividing by $p$ we arrive at 
\begin{equation*}
    f^{*,p}((1-\lambda)y_1 + \lambda y_2)\leq (1-\lambda)f^{*,p}(y_1) + \lambda f^{*,p}(y_2), 
\end{equation*}
demonstrating convexity.     
\end{proof}

\subsubsection{Lower semi-continuity}
\begin{lemma}
\label{lowerSemiContLemma}
    Let $p\in (0,\infty)$. For $f:\R^n\to \R\cup\{\infty\}$, $f^{*,p}$ is lower semi-continuous. 
\end{lemma}
\begin{proof}
    Fix $y_0\in \R^n$. 
    By Fatou's lemma \cite[\S 2.18]{Folland99} and the monotonicity of the logarithm, 
    \begin{equation*}
        \liminf_{y\to y_0} f^{*,p}(y)\geq \frac1p \log\int_{\R^n}\liminf_{y\to y_0} e^{p\langle x,y\rangle - (p+1)f(x)}\frac{dx}{V(f)} = f^{*,p}(y_0), 
    \end{equation*}
    hence $f^{*,p}$ is lower semi-continuous. 
\end{proof}

\subsubsection{Smoothness}
\begin{proposition}
\label{LpLegendreSmooth}
    Let $p\in (0,\infty)$. For $\phi\in\Cvx(\R^n)$, $\phi^{*,p}$ is a smooth function in $\mathrm{int}\,\{\phi^{*,p}<\infty\}$. 
\end{proposition}

Proposition \ref{LpLegendreSmooth} follows directly from the following: 
\begin{lemma}
    Let $p\in (0, \infty)$ and $\phi\in\Cvx(\R^n)$. For $m \in \mathbb{N}$, $m_1, \ldots, m_n\geq 0$ integers with $m_1 + \ldots + m_n= m$, and $y\in \mathrm{int}\{\phi^{*,p}<\infty\}$, 
    \begin{equation*}
        \frac{\partial^m}{\partial y_1^{m_1} \ldots \partial y_n^{m_n}}\int_{\R^n}e^{p\langle x,y\rangle - (p+1)\phi(x)}dx=p^m \int_{\R^n}x_1^{m_1}\ldots x_n^{m_n} e^{p\langle x,y\rangle -(p+1)\phi(x)}dx
    \end{equation*}
\end{lemma}
\begin{proof}
    We do the $m=1$ case; the higher-order cases follow similarly. Let $e_1, \ldots, e_n$ be the standard basis in $\R^n$. Fix $y\in \mathrm{int}\,\{\phi^{*,p}<\infty\}$ and choose $r>0$ such that $y+ rB_2^n\subset \mathrm{int}\{\phi^{*,p}<\infty\}$. For $0<\e<r$,  
    \begin{equation}
    \label{smoothnessEq1}
        \frac1\e\left(\int_{\R^n}e^{p\langle x, y+ \e e_i\rangle} e^{-(p+1)\phi(x)}dx - \int_{\R^n}e^{p\langle x,y\rangle -(p+1)\phi(x)}dx\right)= \int_{\R^n}\frac{e^{p\e x_i}-1}{\e} e^{p\langle x,y\rangle- (p+1)\phi(x)}dx.  
    \end{equation}
    Since $|\frac{e^{\e p x_i}-1}{\e}|\leq \frac1r e^{pr|x_i|}$, we can bound the integral \eqref{smoothnessEq1}, 
    \begin{equation*}
    \begin{aligned}
        \frac1r \int_{\R^n} e^{pr|x_i|} e^{p\langle x,y\rangle - (p+1)\phi(x)}dx
        &=\frac1r \int_{\R^{n-1}}\left(\int_{-\infty}^0 e^{-prx_i} + \int_{0}^\infty e^{prx_i}  \right) e^{p\langle x,y\rangle - (p+1)\phi(x)}dx \\
        &\leq \frac1r \int_{\R^n} e^{p\langle x,y - re_i\rangle- (p+1)\phi(x)}dx + \frac1r \int_{\R^n} e^{p\langle x, y + re_i\rangle - (p+1)\phi(x)}dx 
    \end{aligned}
    \end{equation*}
    which is finite because $y +rB_2^n\subset \mathrm{int}\,\{\phi^{*,p}<\infty\}$. Therefore, Dominated Convergence applies to \eqref{smoothnessEq1} to conclude: 
    \begin{equation*}
    \begin{aligned}
        \frac{\partial}{\partial y_i}\int_{\R^n} e^{p\langle x,y\rangle-(p+1)\phi(x)}dx &= \lim_{\e\to 0}\frac1\e \left( \int_{\R^n}e^{p\langle x,y+ \e e_i\rangle- (p+1)\phi(x)} - \int_{\R^n} e^{p\langle x,y\rangle- (p+1)\phi(x)}dx\right) \\
        &= \int_{\R^n} \lim_{\e\to 0}\frac{e^{p\e x_i}-1}{\e} e^{p\langle x,y\rangle - (p+1)\phi(x)}dx \\
        &= p \int_{\R^n}x_i e^{p\langle x,y\rangle - (p+1)\phi(x)}dx.
    \end{aligned}
    \end{equation*}
    This integral is finite because $y\in \mathrm{int}\,\{\phi^{*,p}<\infty\}$ implies $V((p+1)\phi- p\langle \,\cdot\,, y\rangle)<\infty$, thus $(p+1)\phi- p\langle \,\cdot\,, y\rangle \in \Cvx(\R^n)$, and hence all its moments are finite by Lemma \ref{FinitenessofMoments}.  
\end{proof}

\subsubsection{Additional properties}
One of the properties that does not carry over to the functional case is
\begin{equation*}
    h_{p,K}(y) = \frac1p h_{1,K}(py)
\end{equation*}
relating $h_{p,K}$ to $h_{1,K}$ for any index $p>0$ \cite[Lemma 2.1 (i)]{BMR23}. However, the remaining properties of $h_{p,K}$ continue to hold \cite[Lemma 2.1]{BMR23}. 

The applications of a linear transformation $A\in GL(n,\R)$ to $K$ corresponds to pulling back its convex indicator function by $A^{-1}$. Indeed, $\bm{1}_{AK}^\infty(x)= 1$ if and only if $x\in AK$, or equivalently, $A^{-1}x\in K$, so $\bm{1}^\infty_{K}(A^{-1}x)=1$. For $A\in GL(n,\R)$ denote by 
\begin{equation*}
    (A^*f)(x)\defeq f(Ax), \quad x \in \R^n. 
\end{equation*}
the pull-back of $f$ by the linear transformation $A$. 

\begin{lemma}
\label{listLemma}
    Let $0<p < q<\infty$. For functions $f,g :\R^n \to \R\cup\{\infty\}$ with $0<V(f), V(g)<\infty$, and $A\in GL(n,\R), a\in \R^n$: 

    \smallskip
    \noindent
    (i) $(T_a f)^{*,p}(y) = f^{*,p}(y) - \langle a, y\rangle$.  

    \smallskip 
    \noindent
    (ii) $(A^*f)^{*,p}(y) = f^{*,p}((A^{-1})^T y)$

    \smallskip 
    \noindent
    (iii) $f^{*,p}\leq f^{*,q}\leq f^*$. 

    \smallskip
    \noindent
    (iv) If $f\leq g$, then $g^{*,p}\leq f^{*,p} + \frac1p \log\frac{V(f)}{V(g)}.$
    
\end{lemma}
\begin{proof}
    (i) Note that $V(T_a f) = V(f)$. Using the change of variables $z= x+a$, 
    \begin{equation*}
        \begin{aligned}
            (T_af)^{*,p}(y) &= \frac1p \log \int_{\R^n} e^{p(\langle x,y\rangle - f(x+a))} \frac{e^{-f(x+a)}dx}{V(f)} \\
            &= \frac1p \log\int_{\R^n} e^{p(\langle z-a, y\rangle - f(z))} \frac{e^{-f(z)}dz}{V(f)} \\
            &= \frac1p \log\left( \int_{\R^n}e^{p(\langle z,y\rangle - f(z))} \frac{e^{-f(z)}dz}{V(f)} e^{-p\langle a,y\rangle}\right) \\
            &= f^{*,p}(y) - \langle a, y\rangle. 
        \end{aligned}
    \end{equation*}

    \smallskip
    \noindent
    (ii) For the volume of $A^*f$ we find 
    \begin{equation}
    \label{VolLinearTransformEq}
        V(A^*f) = \int_{\R^n} e^{-f(Ax)}dx = \int_{\R^n} e^{-f(z)} \frac{dz}{|\det\, A|}= \frac{V(f)}{|\det\,A|}. 
    \end{equation}
    Therefore, using the same change of variables $z= Ax, dz= |\det\,A|dx$, by \eqref{VolLinearTransformEq},  
    \begin{equation*}
        \begin{aligned}
            (A^*f)^{*,p}(y)&= \frac1p \log\int_{\R^n} e^{p(\langle x, y\rangle - f(Ax))}\frac{e^{-f(Ax)}dx}{V(A^*f)} \\
            &= \frac1p \log\int_{\R^n} e^{p(\langle A^{-1}z, y\rangle - f(z))} \frac{e^{-f(z)}dz}{|\det \,A|V(A^* f)} \\
            &= \frac1p \log\int_{\R^n} e^{p(\langle z, (A^{-1})^T y\rangle - f(z))}\frac{e^{-f(z)}dz}{V(f)} \\
            &= f^{*,p}((A^{-1})^T y). 
        \end{aligned}
    \end{equation*}

    \smallskip
    \noindent
    (iii) By definition of the Legendre transform, $\langle x,y\rangle - f(x)\leq f^*(y)$ for all $x,y\in \R^n$. Therefore, 
    \begin{equation*}
        f^{*,q}(y)\defeq \frac1q \log\int_{\R^n} e^{q(\langle x,y\rangle - f(x))} \frac{e^{-f(x)}dx}{V(f)}\leq \frac1q \log\int_{\R^n} e^{qf^*(y)}\frac{e^{-f(x)}dx}{V(f)} = \frac1q\log e^{qf^*(y)} = f^*(y).
    \end{equation*}
    In addition, for $p<q$, by H\"older's inequality, 
    \begin{equation*}
        \begin{aligned}
            f^{*,p}(y)&\defeq \frac1p \log\int_{\R^n} e^{p(\langle x,y\rangle - f(x))}\frac{e^{-f(x)}}{V(f)} = \frac1p \log\int_{\R^n} e^{\frac{p}{q}q(\langle x,y\rangle - f(x))}\frac{e^{-f(x)}}{V(f)} \\
            &= \frac1p \log\int_{\R^n}\left( e^{q(\langle x,y\rangle - f(x))}\frac{e^{-f(x)}}{V(f)}\right)^{\frac{p}{q}} \left( \frac{e^{-f(x)}}{V(f)}\right)^{1-\frac{p}{q}}dx \\
            &\leq \frac1p \log\left[ \left( \int_{\R^n}e^{q(\langle x,y\rangle- f(x))}\frac{e^{-f(x)}dx}{V(f)}\right)^{\frac{p}{q}} \left( \int_{\R^n}\frac{e^{-f(x)}}{V(f)}dx\right)^{1-\frac{p}{q}}\right] \\
            &= \frac1q \log\int_{\R^n} e^{q(\langle x,y\rangle - f(x))} \frac{e^{-f(x)}dx}{V(f)} = f^{*,q}(y).
        \end{aligned}
    \end{equation*}

    \smallskip
    \noindent
    (iv) Since $f\leq g$, $e^{-(p+1)g(x)}\leq e^{-(p+1)f(x)}$ hence
    \begin{equation*}
        e^{pg^{*,p}(y)}= \int_{\R^n}e^{p\langle x,y\rangle - (p+1)g(x)}\frac{dx}{V(g)} \leq \int_{\R^n} e^{p\langle x,y\rangle- (p+1)f(x)}\frac{dx}{V(f)} \frac{V(f)}{V(g)} = e^{pf^{*,p}(y)} \frac{V(f)}{V(g)}. 
    \end{equation*}
    Applying $\log$ to both sides proves the claim. 
\end{proof}

\subsubsection{A reverse inequality}
By Lemma \ref{listLemma}(iii), for $p\leq q$,
\begin{equation*}
    f^{*,p}\leq f^{*,q},
\end{equation*}
A reverse inequality holds for convex functions that have their barycenter at the origin. For a function $f:\R^n\to \R\cup\{\infty\}$ with $0<V(f)<\infty$, let
\begin{equation*}
    b(f)\defeq \int_{\R^n}x e^{-f(x)}\frac{dx}{V(f)}
\end{equation*}
denote its barycenter. 
\begin{lemma}
    \label{ReverseIneq}
    Let $0<p<q\leq \infty$. For $\phi\in\Cvx(\R^n)$ with $b(\phi)=0$, 
    \begin{equation*}
        \phi^{*,q}(y)\leq \phi^{*,p}\left( \frac{1+p}{p} y\right) + \frac{n}{p}\log(1+p). 
    \end{equation*}
\end{lemma}
\begin{proof}
    Let $\lambda\in (0,1)$ to be chosen later. Fix $x, y\in \R^n$. Since $b(\phi)=0$, we may write $\lambda x = (1-\lambda)b(\phi) + \lambda x$. In addition, $x= \int_{\R^n}xe^{-\phi(u)}\frac{du}{V(\phi)}$, which allows us to write $\lambda x$ as: 
    \begin{equation}
    \label{ReverseIneqEq1}
    \begin{aligned}
        \lambda x &= \int_{\R^n}\Big( (1-\lambda) u + \lambda x\Big)e^{-\phi(u)}\frac{du}{V(\phi)}. 
    \end{aligned}
    \end{equation}
    By \eqref{ReverseIneqEq1} and Jensen's inequality for the probability measure $\frac{e^{-\phi(u)}du}{V(\phi)}$, 
    \begin{equation}
    \label{ReverseIneqEq2}
        \begin{aligned}
            e^{\langle x,y \rangle}&= e^{\langle \lambda x, \frac{y}{\lambda}\rangle} = e^{\int_{\R^n} \langle ((1-\lambda)u + \lambda x), \frac{y}{\lambda}\rangle \frac{e^{-\phi(u)}du}{V(\phi)}} 
            \leq \int_{\R^n} e^{\langle (1-\lambda)u + \lambda x, \frac{y}{\lambda}\rangle} \frac{e^{-\phi(u)} du}{V(\phi)}. 
        \end{aligned}
    \end{equation}
    Apply the change of variables $v= (1-\lambda)u + \lambda x, dv = (1-\lambda)^n du$. By convexity, $\phi(v)\leq (1-\lambda)\phi(u) + \lambda \phi(x)$, hence $\phi(u)\geq \frac{\phi(v)}{1-\lambda} - \frac{\lambda\phi(x)}{1-\lambda}$. Consequently, \eqref{ReverseIneqEq2} becomes: 
    \begin{equation*}
        e^{\langle x,y\rangle}
        \leq \frac{1}{(1-\lambda)^n} \int_{\R^n} e^{\langle v, \frac{y}{\lambda}\rangle} e^{-\frac{\phi(v)}{1-\lambda} + \frac{\lambda\phi(x)}{1-\lambda}} \frac{dv}{V(\phi)} 
        = \frac{e^{\frac{\lambda}{1-\lambda}\phi(x)}}{(1-\lambda)^n} \int_{\R^n} e^{\langle v, \frac{y}{\lambda}\rangle} e^{-\frac{\phi(v)}{1-\lambda}} \frac{dv}{V(\phi)}
    \end{equation*}
    Choosing $\lambda = \frac{p}{1+p}$,
    \begin{equation*}
        e^{\langle x,y\rangle}\leq (p+1)^n e^{p\phi(x)} \int_{\R^n} e^{\langle v, \frac{p+1}{p} y\rangle}e^{-(p+1)\phi(v)}\frac{dv}{V(\phi)} 
        = (1+p)^n e^{p\phi(x)} e^{p \phi^{*,p}(\frac{1+p}{p^2} y)}. 
    \end{equation*}
    Replacing $y$ by $py$, 
    \begin{equation*}
        e^{p\langle x,y\rangle- p\phi(x)}\leq (1+p)^n  e^{p\phi^{*,p}(\frac{p+1}{p}y)}
    \end{equation*}
    Raising both sides to the $q/p$-th power and then integrating against $\frac{e^{-\phi(x)}dx}{V(\phi)}$, 
    \begin{equation*}
        e^{q \phi^{*,q}(y)}= \int_{\R^n}e^{q(\langle x,y\rangle - \phi(x))}\frac{e^{-\phi(x)}dx}{V(\phi)}\leq (1+p)^{\frac{qn}{p}} e^{q \phi^{*,p}(\frac{1+p}{p}y)}
    \end{equation*}
    because the right-hand side is independent of $x$. Taking $\log$ proves the claim. 
\end{proof}

\begin{remark}
    Lemma \ref{ReverseIneq} differs from the corresponding geometric lemma \cite[Lemma 2.6]{BMR23} in that for convex bodies, due to the fact that $\lambda \bm{1}_K^\infty = \bm{1}_K^\infty$ for all $\lambda$, it is possible to show 
    \begin{equation*}
        h_K(y)\leq h_{p,K}(y/\lambda) - \frac{n}{p}\log(1-\lambda), \quad y\in \R^n,
    \end{equation*}
    when $b(K)=0$. In the functional case, we can only show the inequality for $\lambda = p/(p+1)$. 
\end{remark}

For the convex functions in the class $\Cvx(\R^n)$, i.e., those with non-zero and finite volume, the barycenter always exists (Corollary \ref{FinitenessofMoments}). More generally, all the moments of the function are finite \cite[Lemma 2.2.1]{BGV+14}. Therefore, Lemma \ref{ReverseIneq} can be applied on any convex function after translating by the barycenter. 
\begin{corollary}
    \label{ReverseIneq2}
    Let $0<p<q\leq \infty$. For $\phi\in\Cvx(\R^n)$, 
    \begin{equation*}
        \phi^{*,q}(y)\leq \phi^{*,p}\left( \frac{1+p}{p} y\right) - \frac{\langle b(\phi), y\rangle}{p} + \frac{n}{p}\log(1+p). 
    \end{equation*}
\end{corollary}
\begin{proof}
    Let $\psi(y)\defeq (T_{b(\phi)}\phi)(y)= \phi(y+b(\phi))$, so that $b(\psi) = 0$. By Lemma \ref{ReverseIneq} and Lemma \ref{listLemma}(i), 
    \begin{equation*}
    \begin{aligned}
        \phi^{*,q}(y)-\langle b(\phi), y\rangle = \psi^{*,p}(x)
        &\leq \psi^{*,p}\left( \frac{1+p}{p}y\right) + \frac{n}{p}\log(1+p) \\
        &= \phi^{*,p}\left( \frac{1+p}{p} y\right) -\frac{1+p}{p}\langle b(\phi), y\rangle + \frac{n}{p}\log(1+p).
    \end{aligned}
    \end{equation*}
    Adding $\langle b(\phi), y\rangle$ to both sides proves the claim. 
\end{proof}

\begin{corollary}
\label{LpConvergence}
    For $\phi\in\Cvx(\R^n)$ and $y\in \R^n$, $\lim_{p\to\infty}\phi^{*,p}(y)= \phi^{*}(y)$. 
\end{corollary}
\begin{proof}
    By Lemma \ref{listLemma}(iii), $\{\phi^{*,p}(y)\}_{p> 0}$ is monotone increasing and bounded from above by $\phi^{*}(y)$. Therefore, the limit exists and $\lim_{p\to\infty}\phi^{*,p}(y)\leq \phi^*(y)$. On the other hand, by Corollary \ref{ReverseIneq2}, 
    \begin{equation*}
        \phi^{*,p}(y)\geq \phi^*\left(\frac{p}{p+1}y\right) + \frac{\langle b(\phi), y\rangle}{p+1} -\frac{n}{p}\log(p+1)
    \end{equation*}
    Using the lower semi-continuity of $\phi^{*}$ (Lemma \ref{lowerSemiContLemma}), 
    \begin{equation*}
        \lim_{p\to\infty}\phi^{*,p}(y)\geq \liminf_{p\to\infty}\phi^{*}\left( \frac{p}{p+1}y\right)\geq \phi^{*}\left( \lim_{p\to\infty}\frac{p}{p+1}y\right)= \phi^{*}(y), 
    \end{equation*}
    as desired. 
\end{proof}

\subsubsection{Additional convexity properties}
In Lemma \ref{LpLegendreConvexity}, we demonstrated the convexity of the $L^p$-Legendre transform $f^{*,p}(y)$ as a function of $y$. In this brief section, we document convexity properties of the transform as a function of $p$ and $f$. Although these two lemmas are not used elsewhere in the article, we include them for completeness, as they generalize the corresponding geometric results \cite[Lemmas 2.4--2.5]{BMR23}. 
\begin{lemma}
\label{pConvLemma}
    Let $p,q\in (0,\infty)$. For $f:\R^n\to \R\cup\{\infty\}$ and $\lambda\in (0,1)$, 
    \begin{equation*}
        f^{*, (1-\lambda)p + \lambda q}(y) \leq \frac{(1-\lambda)p}{(1-\lambda)p + \lambda q} f^{*,p}(y) + \frac{\lambda q}{(1-\lambda)p + \lambda q} f^{*,q}(y). 
    \end{equation*}
\end{lemma}
\begin{proof}
    By H\"older's inequality, 
    \begin{equation*}
        \begin{aligned}
            &f^{*, (1-\lambda)p + \lambda q}(y)\\
            &= \frac{1}{(1-\lambda)p + \lambda q}\log\int_{\R^n}e^{((1-\lambda)p + \lambda q)(\langle x, y\rangle - f(x))} \frac{e^{-f(x)}dx}{V(f)} \\
            &= \frac{1}{(1-\lambda)p + \lambda q} \log\int_{\R^n}\left( e^{p(\langle x,y\rangle - f(x))}\frac{e^{-f(x)}}{V(f)}\right)^{1-\lambda} \left( e^{q(\langle x, y\rangle - f(x))}\frac{e^{-f(x)}}{V(f)}\right)^\lambda dx \\
            &\leq \frac{1}{(1-\lambda)p + \lambda q} \log\left[\left(\int_{\R^n} e^{p(\langle x,y\rangle - f(x))}\frac{e^{-f(x)}dx}{V(f)}\right)^{1-\lambda} \left(\int_{\R^n} e^{q(\langle x, y\rangle - f(x))}\frac{e^{-f(x)}dx}{V(f)}\right)^\lambda \right] \\
            &= \frac{(1-\lambda)p}{(1-\lambda)p + \lambda q} \frac1p \log\int_{\R^n}e^{p(\langle x,y\rangle - f(x))} \frac{e^{-f(x)}}{V(f)}dx + \frac{\lambda q}{(1-\lambda)p + \lambda q}\frac1q \log\int_{\R^n}e^{q(\langle x,y\rangle - f(x))}\frac{e^{-f(x)}dx}{V(f)} \\
            &= \frac{(1-\lambda)p}{(1-\lambda)p + \lambda q} f^{*,p}(y) + \frac{\lambda q}{(1-\lambda)p + \lambda q} f^{*,q}(y). 
        \end{aligned}
    \end{equation*}
\end{proof}

The infimum convolution
\begin{equation*}
    (f\,\square\, g)(z)\defeq \inf\{f(x) + g(z-x): x\in\R^n\} 
\end{equation*}
generalizes the Minkowski sum of convex bodies. For example, for the convex indicator function, 
\begin{equation*}
    (\bm{1}_K^\infty\,\square\,\bm{1}_L^\infty)(z) = \inf\{\bm{1}_K^\infty(x)+ \bm{1}_Y^\infty(z-x): x\in \R^n\} =\begin{dcases}
        0, & \text{ if there exists } x\in K \text{ with } z-x\in L, \\
        \infty, & \text{ otherwise},
    \end{dcases}
\end{equation*}
or equivalently, $(\bm{1}_K^\infty\,\square\,\bm{1}_L^\infty)(z) = 0$ if and only if $z\in K+ L$, which implies $(\bm{1}_K^\infty\,\square\,\bm{1}_L^\infty)(z)= \bm{1}^\infty_{K+L}$. 
The Pr\'ekopa--Leindler \cite[Theorem 3]{Prekopa73} gives
\begin{equation*}
    V(\big(((1-\lambda)\cdot f)\,\square\, (\lambda\cdot g)\big)) \geq V(f)^{1-\lambda} V(g)^\lambda,
\end{equation*}
so if both $f$ and $g$ have non-zero volume, so does $((1-\lambda)\cdot f)\,\square\, (\lambda\cdot g)$. 

To extend the rescaling of a convex body to functions, define: 
\begin{equation*}
    (\lambda\cdot f)(x)\defeq \lambda f\left(\frac{x}{\lambda}\right), \quad x\in \R^n. 
\end{equation*}
This preserves convexity and corresponds exactly to rescaling for convex bodies since $(\lambda\cdot \bm{1}_K^\infty)(x)= \lambda\bm{1}_K^\infty(x/\lambda)$ equals 0 if and only if $x/\lambda\in K$ or $x\in \lambda K$. That is, $(\lambda\cdot\bm{1}_K^\infty)= \bm{1}_{\lambda K}^\infty$. 
\begin{lemma}
\label{fConvLemma}
    Let $p\in (0,\infty)$. For $f,g:\R^n\to \R\cup\{\infty\}$ and $\lambda\in (0,1)$, 
    \begin{equation*}
        \big(((1-\lambda)\cdot f)\,\square\, (\lambda\cdot g)\big)^{*,p} \geq (1-\lambda) f^{*,p} + \lambda g^{*,p} -\frac1p\log\left(\frac{V(\big(((1-\lambda)\cdot f)\,\square\, (\lambda\cdot g)\big))}{V(f)^{1-\lambda}V(g)^\lambda}\right). 
    \end{equation*}
\end{lemma}
\begin{proof}
    For $x,z\in \R^n$
    \begin{equation}
    \label{infConvEq1}
        \begin{aligned}
             &\big(((1-\lambda)\cdot f)\,\square\, (\lambda\cdot g)\big)((1-\lambda) x + \lambda z))\\
             &\defeq \inf_{w\in \R^n}\left(((1-\lambda)\cdot f)(w) + (\lambda\cdot g)((1-\lambda)x + \lambda z - w) \right) \\
             &\leq ((1-\lambda)\cdot f)((1-\lambda)x) + (\lambda\cdot g)((1-\lambda) x + \lambda z - (1-\lambda)x) \\
             &= (1-\lambda)f(x) + \lambda g(z). 
        \end{aligned}
    \end{equation}
    Fix $y\in\R^n$. By \eqref{infConvEq1}, for all $x,z\in \R^n$, 
    \begin{equation*}
        \begin{aligned}
            &p\langle (1-\lambda)x + \lambda z, y\rangle - (p+1)\big(((1-\lambda)\cdot f)\,\square\, (\lambda\cdot g)\big)((1-\lambda) x + \lambda z)) \\
            &= (1-\lambda) p\langle x, y\rangle+ \lambda p \langle z,y\rangle - (p+1)\big(((1-\lambda)\cdot f)\,\square\, (\lambda\cdot g)\big)((1-\lambda) x + \lambda z))\\
            &\geq p(1-\lambda)\langle x,y\rangle + p\lambda \langle z, y\rangle - (p+1)(1-\lambda) f(x) - (p+1) \lambda g(z) \\
            &= (1-\lambda)(p\langle x, y\rangle - (p+1)f(x)) + \lambda(p\langle z, y\rangle - (p+1) g(z))
        \end{aligned}
    \end{equation*}
    Consequently, by the Pr\'ekopa--Leindler inequality \cite[Theorem 3]{Prekopa73}, 
    \begin{equation*}
    \begin{aligned}
        &\int_{\R^n} e^{p\left(\langle x, y\rangle - \big(((1-\lambda)\cdot f)\,\square\, (\lambda\cdot g)\big)(x)\right)} e^{-\big(((1-\lambda)\cdot f)\,\square\, (\lambda\cdot g)\big)} dx \\
        &\geq \left( \int_{\R^n}e^{p(\langle x,y\rangle - f(x))} e^{-f(x)}dx\right)^{1-\lambda} \left(\int_{\R^n}e^{p(\langle z, y\rangle - g(x))} e^{-g(x)}dx\right)^\lambda. 
    \end{aligned}
    \end{equation*}
    Dividing both sides by $V(\big(((1-\lambda)\cdot f)\,\square\, (\lambda\cdot g)\big))$ and applying $\log$ proves the claim. 
\end{proof}

\subsubsection{Tensoriality}
For functions $f:\R^n\to \R\cup\{\infty\}$ and $g:\R^m\to \R\cup\{\infty\}$, denote by 
\begin{equation*}
    (f\otimes g)(x, y)\defeq f(x) + g(y), \quad (x, y)\in\R^n\times \R^m, 
\end{equation*}
generalizing the Cartesian product of convex bodies: $\bm{1}^\infty_K \otimes \bm{1}_L^\infty = \bm{1}_{K\times L}^\infty$. 
Let, also, 
\begin{equation*}
    f^{\otimes m}(x_1, \ldots, x_m)\defeq f(x_1) + \ldots + f(x_m), \quad x_1, \ldots, x_m\in \R^n.  
\end{equation*}
It follows from Tonelli's Theorem \cite[\S 2.37]{Folland99} that 
\begin{equation}
\label{tensorVolume}
    V(f\otimes g) = \int_{\R^n\times \R^m}e^{-f(x)-g(y)}dxdy= \int_{\R^n}e^{-f(x)}dx \int_{\R^m}e^{-g(y)}dy= V(f) V(g). 
\end{equation}

\begin{lemma}
\label{LpLegendreTensor}
    Let $p\in (0,\infty)$. For $f:\R^n\to \R\cup\{\infty\}$ and $g: \R^m\to \R\cup\{\infty\}$ with $0< V(f), V(g)<\infty$,
    \begin{equation*}
        (f\otimes g)^{*,p}(x,y) = f^{*,p}(x) + g^{*,p}(y), \quad (x, y)\in\R^n\times \R^m. 
    \end{equation*}
\end{lemma}
\begin{proof}
    Using the definition of $L^p$-Legendre transform and \eqref{tensorVolume}, 
    \begin{equation*}
        \begin{aligned}
            (f\otimes g)^{*, p}(x,y) &= \frac1p \log\int_{\R^n\times \R^m}e^{p(\langle (x, y), (u,v)\rangle - (f\otimes g)(u, v))}\frac{e^{-(f\otimes g)(u,v)} du dv}{V(f\otimes g)}\\
            &= \frac1p \log\int_{\R^n\times \R^m}e^{p(\langle x, u\rangle + \langle y,v\rangle - f(u)- g(v))}\frac{e^{-f(u) - g(v)} du dv}{V(f)V(g)} \\
            &= \frac1p \log \left(\int_{\R^n} e^{p(\langle x,u\rangle-f(u))} \frac{e^{-f(u)}du}{V(f)} \int_{\R^m}e^{p(\langle y, v\rangle-g(v))}\frac{e^{-g(v)}dv}{V(g)}\right) \\
            &= \frac1p \log \left(\int_{\R^n} e^{p(\langle x,u\rangle-f(u))} \frac{e^{-f(u)}du}{V(f)}\right) +\frac1p\log\left(\int_{\R^m}e^{p(\langle y, v\rangle-g(v))}\frac{e^{-g(v)}dv}{V(g)}\right) \\
            &= f^{*,p}(x) + g^{*,p}(y).
        \end{aligned}
    \end{equation*}
\end{proof}

\subsection{Examples}
\label{LpLegendreExamples}
We now compute the $L^p$-Legendre transform for a few basic examples, starting with the $L^1$ norm. For $q\in (0,\infty]$ and $x\in \R^n$, let
\begin{equation*}
    \|x\|_q\defeq \begin{dcases}(|x_1|^q + \ldots + |x_n|^q)^{1/q},  &q\in (0, \infty),\\
    \max\{|x_1|, \ldots, |x_n|\}, & q=\infty, 
    \end{dcases}
\end{equation*}
Ideally, we would like to compute the $L^p$-Legendre transforms of each $L^q$-norm. Although it is possible to compute $V(\|x\|_q)$ for all $q$, at this stage we can only compute the $L^p$-Legendre transforms of the $L^1$ norm. This is because the $L^1$-norm enjoys a product structure the the other norms lack.

\subsubsection{\texorpdfstring{$L^1$-norm}{L1-norm}}
Denote by 
\begin{equation*}
    \phi_1(x)\defeq \|x\|_1 = |x_1| + \ldots + |x_n|, \quad x\in \R^n. 
\end{equation*}
Computing the volume reduces to a 1-dimensional integral: 
\begin{equation}
\label{L1Volume}
    V(\phi_1)= \int_{\R^n}e^{-|x_1|-\ldots-|x_n|}dx = \left( \int_{\R}e^{-|x|}dx\right)^n= \left( 2\int_0^\infty e^{-x}dx\right)^n = 2^n. 
\end{equation}
\begin{lemma}
\label{LpLegendreL1norm}
    For $p\in (0,\infty)$, 
    \begin{equation*}
        \phi^{*,p}(y)= 
        \begin{dcases}
        -\frac{n}{p}\log(p+1)- \frac1p \sum_{i=1}^n \log\left[1- \left(\frac{p}{p+1} y_i\right)^2\right], & \|y\|_\infty < \frac{p+1}{p}\\
        \infty, & \|y\|_{\infty}\geq \frac{p+1}{p}.  
        \end{dcases}
    \end{equation*}
\end{lemma}
\begin{proof}
    By Lemma \ref{LpLegendreTensor} it is enough to prove the claim for $n=1$. In this case, by \eqref{L1Volume}, 
    \begin{equation}
    \label{L1eq1}
        \phi_1^{*,p}(y)= \frac1p \log\int_{\R}e^{pxy - (p+1)|x|}dx = \frac1p\log\left[ \int_{-\infty}^0 e^{pxy+ (p+1)x}\frac{dx}{2}+ \int_0^\infty e^{pxy - (p+1)x}\frac{dx}{2}\right]. 
    \end{equation}
    For the first integral, 
    \begin{equation}
    \label{L1eq2}
        \int_{-\infty}^0 e^{x(py + (p+1))}dx =\begin{dcases}
           \frac{1}{py + p + 1} & y > -\frac{p+1}{p}, \\
           \infty  & y\leq -\frac{p+1}{p}. 
        \end{dcases}
    \end{equation}
    For the second integral, 
    \begin{equation}
    \label{L1eq3}
        \int_0^\infty e^{x(py- (p+1))}\frac{dx}{2}= \begin{dcases}
           \frac{1}{p+1 - py} & y < \frac{p+1}{p}, \\
           \infty  & y\geq \frac{p+1}{p}. 
        \end{dcases}
    \end{equation}
    By \eqref{L1eq1}--\eqref{L1eq3}, $\phi^{*,p}(y)= \infty$ for $|y|\geq \frac{p+1}{p}$. For $|y|<\frac{p+1}{p}$, 
    \begin{equation*}
        \phi^{*,p}(y)=\frac1p\log\left[\frac{1}{2(p+1 + py)}+\frac{1}{2(p+1-py)}\right]= \frac1p\log\frac{p+1}{(p+1)^2 - (py)^2}
    \end{equation*}
    Dividing both numerator and denominator by $(p+1)^2$, $\phi^{*,p}(y) = \frac1p \log\frac{(p+1)^{-1}}{1 - (\frac{p}{p+1} y)^2}= - \frac1p\log(p+1) - \frac1p\log[1- (\frac{p}{p+1}y)^2]$. 
\end{proof}

\begin{remark}
    Sending $p\to \infty$ in Lemma \ref{LpLegendreL1norm}, $\lim_{p\to\infty}\frac{\log(1+p)}{p}= 0$, and by L'H\^opital, for $\|y\|_\infty\leq 1$, 
    \begin{equation*}
        \lim_{p\to\infty}\frac1p \log\left[1- \left(\frac{p}{p+1}y_i\right)^2\right]= \lim_{p\to\infty}\frac{-2\frac{p}{p+1}y_i \frac{1}{(1+p)^2}y_i}{1- (\frac{p}{p+1} y_i)^2}= 0.  
    \end{equation*}
    That is, $\phi_1^{*}= \bm{1}_{[-1,1]^n}^\infty$ as we already know. For example, using duality of the Legendre transform, the support function of the cube: $h_{[-1,1]^n}(y)= (\bm{1}_K^\infty)^*(y)= \phi_1(y)= \|y\|_1$, equals the $L^1$-norm. 
\end{remark}

\subsubsection{Convex quadratic}
Denote by 
\begin{equation*}
    q_2(x)\defeq |x|^2/2, \quad x\in \R^n. 
\end{equation*}
\begin{lemma}
\label{ConvexQuadraticLpLegendre}
    For $p\in (0,\infty)$, 
    \begin{equation*}
        q_2^{*,p}(y)= \frac{p}{p+1}\frac{|y|^2}{2} - \frac{n}{2p}\log(1+ p), \quad y\in \R^n. 
    \end{equation*}
\end{lemma}
For the proof of Lemma \ref{ConvexQuadraticLpLegendre} we need the following claim: 
\begin{claim}
\label{GaussClaim}
For $A\in GL(n, \R)$ symmetric and positive definite, and $b\in \R^n$, 
\begin{equation*}
    \int_{\R^n} e^{-\frac12 \langle Ax, x\rangle + \langle b,x\rangle}\dif x= \sqrt{\frac{(2\pi)^n}{\det A}} \,e^{\frac12\langle A^{-1}b, b\rangle}. 
\end{equation*}
\end{claim}
\begin{proof}
By the symmetry of $A$, $\langle Ax, y\rangle= \langle x, Ay\rangle$ for all $x,y\in \R^n$. Now, for $x= y+ A^{-1}b$, 
\begin{equation*}
    \begin{aligned}
    -\frac12\langle Ax, x\rangle+ \langle b,x\rangle&= -\frac12\langle Ay+b, y+ A^{-1}b\rangle+ \langle b, y+ A^{-1}b\rangle \\
    &= -\frac12\langle Ay,y\rangle-\frac12\langle Ay,  A^{-1}b\rangle- \frac12\langle b,y\rangle-\frac12\langle b, A^{-1}b\rangle+ \langle b,y\rangle+ \langle b, A^{-1}b\rangle\\
    &= -\frac12\langle Ay, y\rangle -\frac12\langle y,b\rangle-\frac12\langle y,b\rangle-\frac12\langle b, A^{-1}b\rangle+ \langle b,y\rangle+ \langle b, A^{-1}b\rangle \\
    &= -\frac12\langle Ay, y\rangle +\frac12\langle b, A^{-1}b\rangle. 
    \end{aligned}
\end{equation*}
Therefore, changing variables $y= x- A^{-1}b, dy= dx$, 
\begin{equation}
\label{gaussAeq1}
    \int_{\R^n}e^{-\frac12 \langle Ax,x\rangle+ \langle b,x\rangle}\dif x= e^{\frac12\langle b, A^{-1}b\rangle}\int_{\R^n}e^{-\frac12\langle Ay,y\rangle}\dif y. 
\end{equation}
To compute the remaining integral, note that since $A$ is symmetric and positive definite, there exists a symmetric and positive definite matrix $B\in GL(n,\R)$ such that $A= B^2$. Thus, we can rewrite $\langle Ay, y\rangle$ as follows: 
\begin{equation*}
    \langle Ay, y\rangle = \langle B^2 y, y\rangle = \langle By, B^T y\rangle = \langle By, By\rangle = |By|^2. 
\end{equation*}
Making the substitution $z= By$, which gives $dz= (\det\,B)dy= \sqrt{\det\,A} \, dy$, the integral in the right-hand side of \eqref{gaussAeq1} becomes
\begin{equation}
\label{gaussAeq2}
    \int_{\R^n}e^{-\frac12\langle Ay, y\rangle}= \int_{\R^n}e^{-\frac12 |By|^2}dy = \sqrt{\det\,A}\int_{\R^n} e^{-|z|^2/2}dz= \sqrt{\det\,A}\, (2\pi)^{n/2}.     
\end{equation}
Combining \eqref{gaussAeq1} and \eqref{gaussAeq2} proves the claim. 
\end{proof}

\begin{proof}[Proof of Lemma \ref{ConvexQuadraticLpLegendre}]
    By \eqref{|x|^2/2Vol}, 
    \begin{equation*}
        q_2^{*,p}(y)= \frac1p \log\int_{\R^n} e^{p(\langle x,y \rangle - |x|^2/2)} \frac{e^{-|x|^2/2}dx}{(2\pi)^{n/2}} = \frac1p \log\int_{\R^n} e^{-\frac12\langle (p+1) x, x\rangle + \langle py, x\rangle} \frac{dx}{(2\pi)^{n/2}}. 
    \end{equation*}
    Therefore, by Claim \ref{GaussClaim} with $A= (p+1)I_n$ and $b= py$, 
    \begin{equation*}
        q_2^{*,p}(y) = \frac1p\log\left(\sqrt{\frac{(2\pi)^n}{(p+1)^n}} \,e^{\frac12\langle \frac{p}{p+1} y, py\rangle} \frac{1}{(2\pi)^{n/2}}\right) = \frac1p\log\left( \frac{e^{\frac{p^2}{2}\frac{|y|^2}{p+1}}}{(p+1)^{n/2}}\right), 
    \end{equation*}
    from which the claim follows. 
\end{proof}

\subsubsection{Functional simplex}
Denote by 
\begin{equation*}
    \phi_{\Delta_n}(x)\defeq (x_1 + \ldots + x_n) (1+ \bm{1}^\infty_{[-1, \infty)^n}(x))= \begin{dcases}
        x_1 + \ldots + x_n, & x\in[-1, \infty)^n, \\
        \infty, & \text{otherwise}. 
    \end{dcases}
\end{equation*}
what we will refer to as the $n$-dimensional ``functional simplex". This is the conjecture minimizer for the functional non-symmetric Mahler conjecture \cite{FM08b, FM08a}. To compute its volume
\begin{equation}
\label{functionalSimplexVolume}
    V(\phi_{\Delta_n})= \int_{[-1, \infty)}e^{-(x_1 + \ldots + x_n)}dx = \left( \int_{-1}^\infty e^{-t}dt\right)^n = e^n. 
\end{equation}

In contrast to the simplex, the functional simplex enjoys a product structure: 
\begin{equation*}
    \phi_{\Delta_n}(x_1, \ldots, x_n) = \phi_{\Delta_1}(x_1) + \ldots + \phi_{\Delta_1}(x_n). 
\end{equation*}
Therefore, by the tensoriality properties of the $L^p$-Legendre transform (Lemma \ref{LpLegendreTensor}), to compute $\phi_{\Delta_n}^{*,p}$ it is enough to deal with the 1-dimensional case. 
\begin{lemma}
    \label{functionalSimplexLpLegendre}
    For $p\in (0,\infty)$, 
    \begin{equation*}
        \phi_{\Delta_n}^{*,p}(y)= \begin{dcases}
            n- (y_1 + \ldots + y_n) - \frac1p\sum_{i=1}^n \log(p+1-py_i)), & y\in \left(-\infty, \frac{p+1}{p}\right]^n, \\
            \infty, & \text{otherwise}.
        \end{dcases} 
    \end{equation*}
\end{lemma}
\begin{proof}
    Let us start with $n=1$: By \eqref{functionalSimplexVolume}, 
    \begin{equation*}
        \begin{aligned}
            \phi_{\Delta_1}^{*,p}(y) 
            = \frac1p\log\int_{-1}^\infty e^{p(xy - x)} \frac{e^{-x}dx}{e}
            = \frac1p \log\int_{-1}^\infty e^{x(py-(p+1))}\frac{dx}{e}. 
        \end{aligned}
    \end{equation*}
    Therefore, for $py\geq p+1$ the integral evaluates to $\infty$. For $py<p+1$, 
    \begin{equation*}
        \int_{-1}^\infty e^{x(py-(p+1))}\frac{dx}{e}= \frac1e \left[\frac{e^{x(py-(p+1))}}{py-(p+1)} \right]_{x=-1}^\infty = \frac{e^{p+1 - py}}{e(p+1-py)}.
    \end{equation*}
    That is,
    \begin{equation*}
        \phi_{\Delta_1}^{*,p}(y) = \begin{dcases}
            1- y-  \frac{\log(p+1 - py)}{p}, & y < \frac{p+1}{p}, \\
            \infty, & y \geq \frac{p+1}{p}. 
        \end{dcases}
    \end{equation*}
    The higher dimensional cases follow from the one-dimensional case Lemma \ref{LpLegendreTensor}. 
\end{proof}
\begin{remark}
    For $p=\infty$, using L'H\^optial's rule,
    \begin{equation*}
        \phi_{\Delta_1}^*(y) = \begin{dcases}
            \lim_{p\to\infty}\left[1- y -\frac{\log(1+ p (1-y))}{p}\right], &y \leq 1, \\
            \infty, & y > 1
        \end{dcases}
        = \begin{dcases}
           1-y,  &y\leq 1, \\
           \infty, & y > 1, 
        \end{dcases}
    \end{equation*}
    since logarithm grows slower than any polynomial. 
\end{remark}

\subsubsection{The Euclidean ball}
Let 
\begin{equation*}
    I_\alpha(x) \defeq \sum_{m=0}^\infty \left( \frac{x}{2}\right)^{2m+\alpha} \frac{1}{m! \Gamma(m+\alpha+1)}, \quad x\in \R, 
\end{equation*}
denote the modified Bessel function of the first kind \cite[(2) p. 77]{Watson95}. 

\begin{lemma}
\label{hpB2n}
    For $p\in (0,\infty]$,
    \begin{equation*}
        h_{p, B_2^n}(y) = \frac1p \log\left( \Gamma\left( 1 + \frac{n}{2}\right) \left( \frac{2}{p|y|}\right)^{n/2} I_{n/2}(p|y|)\right).
    \end{equation*}
\end{lemma}

The proof of Lemma \ref{hpB2n} is based on the following calculation \cite[10.32.2]{NIST_handbook}: 
\begin{claim}
\label{hpB2nClaim}
    For $n\in\mathbb{N}$ and $a>0$, 
    \begin{equation*}
        \int_{-1}^1 (1-x^2)^{\frac{n-1}{2}} e^{ax}dx = \sqrt{\pi}\,\Gamma\left( \frac{n+1}{2}\right) \left( \frac{2}{a}\right)^{n/2} I_{n/2}(a). 
    \end{equation*}
\end{claim}
\begin{proof}
    First, using the Taylor expansion of the exponential function: 
    \begin{equation}
    \label{B2nEq2}
        \int_{-1}^1 (1-x^2)^{\frac{n-1}{2}} e^{ax}dx  = \sum_{m=0}^\infty \frac{a^m}{m!}\int_{-1}^1 (1-x^2)^{\frac{n-1}{2}} x^m dx = 2\sum_{m=0}^\infty \frac{a^{2m}}{(2m)!}\int_0^1 (1-x^2)^{\frac{n-1}{2}}x^{2m}dx, 
    \end{equation}
    because 
    \begin{equation*}
        \int_{-1}^1 (1-x^2)^{\frac{n-1}{2}} x^m dx = \begin{dcases}
            0, & m = \text{ odd}, \\
            2\int_0^1 (1-x^2)^{\frac{n-1}{2}} x^{m}dx, & m = \text{ even}.  
        \end{dcases}
    \end{equation*}
    To compute the integral on the right-hand side of \eqref{B2nEq2}, use the change of variables $u = x^2, dx= du/ (2\sqrt{u})$, 
    \begin{equation}
    \label{B2nEq5}
    \begin{aligned}
        \int_{0}^1 (1- x^2)^{\frac{n-1}{2}} x^{2m} dx  &= \frac12 \int_0^1 (1-u)^{\frac{n-1}{2}} u^{m-\frac12}du = \frac12 \int_0^1 (1-u)^{\frac{n+1}{2}-1} u^{m+\frac12 - 1}du \\
        &= \frac12 \mathrm{B}\left( \frac{n+1}{2}, m + \frac12\right) = \frac{\Gamma(\frac{n+1}{2}) \Gamma(m + \frac12)}{2\Gamma(m + \frac{n}{2} +1)}.
    \end{aligned}
    \end{equation}
    Combining \eqref{B2nEq2} and \eqref{B2nEq5}, and using $\Gamma(m)\Gamma(m + \frac12)= 2^{1-2m}\sqrt{\pi}\,\Gamma(2m)$, 
    \begin{equation*}
    \begin{aligned}
        \int_{-1}^1 (1-x^2)^{\frac{n-1}{2}} e^{ax} dx &= 2\sum_{m=0}^\infty \frac{a^{2m}}{(2m)!} \frac{\Gamma(\frac{n+1}{2})\Gamma(m+\frac12)}{2\Gamma(m + \frac{n}{2} + 1)} \\
        &= \Gamma\left( \frac{n+1}{2}\right) \sum_{m=0}^\infty \frac{a^{2m}}{(2m)!} \frac{2^{1-2m} \sqrt{\pi}\,(2m-1)!}{(m-1)! \Gamma(m + \frac{n}{2} + 1)} \\
        &= 2\sqrt{\pi}\,\Gamma\left( \frac{n+1}{2}\right) \sum_{m=0}^\infty \left( \frac{a}{2}\right)^{2m} \frac{m}{(2m) \,m!\, \Gamma(m + \frac{n}{2} + 1)} \\
        &= \sqrt{\pi}\, \Gamma\left( \frac{n+1}{2}\right) \left( \frac{2}{a}\right)^{\frac{n}{2}} \sum_{m=0}^\infty \left( \frac{a}{2}\right)^{2m + \frac{n}{2}} \frac{1}{m! \Gamma(m + \frac{n}{2} + 1)} \\
        &= \sqrt{\pi}\,\Gamma\left( \frac{n+1}{2}\right) \left( \frac{2}{a}\right)^{n/2} I_{n/2}(a), 
    \end{aligned}
    \end{equation*}
    finishing the proof. 
\end{proof}
\begin{proof}[Proof of Lemma \ref{hpB2n}]
    Fix $y\in \R^n\setminus \{0\}$. Let $A$ be the orthogonal matrix that takes $y$ to $|y|e_1$. The change of variables $z= A x$ takes $B_2^n$ to itself with $dz= |\det A| dx = dx$, hence
    \begin{equation}
    \label{B2nEq1}
        \int_{B_2^n} e^{p\langle x, y\rangle}dx = \int_{B_2^n} e^{p\langle A^T z, y\rangle}dz = \int_{B_2^n} e^{p\langle z, A y\rangle}dz = \int_{B_2^n} e^{p|y|z_1}dz. 
    \end{equation}
    Using Fubini's theorem, 
    \begin{equation}
    \label{B2nEq4}
        \int_{B_2^n}e^{p|y|z_1}dz = \int_{-1}^1 \int_{\sqrt{1-z_1^2}B_2^{n-1}} e^{p|y|z_1}dz= |B_2^{n-1}|\int_{-1}^1 (1-z_1^2)^{\frac{n-1}{2}} e^{p|y|z_1}dz_1.  
    \end{equation}
    Using Claim \ref{hpB2nClaim} with $a= p|y|$, it follows form \eqref{B2nEq1} and \eqref{B2nEq4} that
    \begin{equation*}
        \begin{aligned}
            \int_{B_2^n} e^{p\langle x,y \rangle} \frac{dx}{|B_2^n|}= \frac{|B_2^{n-1}|}{|B_2^n|} \sqrt{\pi}\, \Gamma\left(\frac{n+1}{2}\right) \left( \frac{2}{p|y|}\right)^{n/2} I_{n/2}(p|y|)
        \end{aligned}
    \end{equation*}
    Using $|B_2^n| = \frac{\pi^{\frac{n}{2}}}{\Gamma(1 + \frac{n}{2})}$, 
    \begin{equation*}
        \frac{|B_2^{n-1}|}{|B_2^n|} \sqrt{\pi}\,\Gamma\left( \frac{n+1}{2}\right)= \frac{\frac{\pi^{\frac{n-1}{2}}}{\Gamma(1+\frac{n-1}{2})}}{\frac{\pi^{\frac{n}{2}}}{\Gamma(1 + \frac{n}{2})}} \sqrt{\pi}\, \Gamma\left( \frac{n+1}{2}\right) = \frac{\Gamma(1 + \frac{n}{2})}{\sqrt{\pi}\Gamma(\frac{n+1}{2})} \sqrt{\pi}\,\Gamma\left( \frac{n+1}{2}\right) = \Gamma\left( 1 + \frac{n}{2}\right). 
    \end{equation*}
    Consequently, 
    \begin{equation*}
        h_{p, B_2^n}(y)= \frac1p\log\int_{B_2^n}e^{p\langle x,y\rangle}\frac{dy}{|B_2^n|} = \frac1p \log\left( \Gamma\left( 1 + \frac{n}{2}\right) \left( \frac{2}{p|y|}\right)^{n/2} I_{n/2}(p|y|)\right). 
    \end{equation*}
\end{proof}

\begin{remark}
    By $I_\nu(x)\sim \frac{1}{\Gamma(\nu+1)} (x/2)^\nu$ $x\to 0^-$ \cite[10.30.1]{NIST_handbook},
    \begin{equation*}
        \lim_{y\to 0^-} \frac1p \log\left[ \Gamma\left(1+\frac{n}{2}\right)\left(\frac{2}{p|y|}\right)^{n/2}  I_{n/2}(p|y|)\right] = \frac1p\log 1=0 = h_{p, B_2^n}(0).
    \end{equation*}
    For $y\neq 0$, using $I_n(x)\sim \frac{e^x}{\sqrt{2\pi x}}$ for $x\to \infty$ \cite[10.30.4]{NIST_handbook}, 
    \begin{equation*}
    \begin{aligned}
        \lim_{p\to\infty}\frac1p \log\left[ \Gamma\left(1+\frac{n}{2}\right)\left(\frac{2}{p|y|}\right)^{n/2}  I_{n/2}(p|y|)\right]&= \lim_{p\to\infty} \frac1p\log\left[ \Gamma\left(1+\frac{n}{2}\right)\left(\frac{2}{p|y|}\right)^{n/2} \frac{e^{p|y|}}{\sqrt{2\pi p}}\right]\\
        &= \lim_{p\to\infty}\log\left[\left(\frac{\Gamma(1+\frac{n}{2}) 2^{n/2}}{|y|^{n/2}\sqrt{2\pi}}\right)^{1/p} p^{-\frac{n+1}{2p}} e^{|y|}\right] \\
        &= \log e^{|y|}= |y|  = h_{B_2^n}(x). 
    \end{aligned}
    \end{equation*}
\end{remark}

\subsubsection{The simplex}
\begin{lemma}
For $p\in (0,\infty)$, 
\label{SimplexLpSupport}
    \begin{equation*}
        h_{p,\Delta_{n,+}}(y) = - \frac{n}{p} \log p + \frac1p\log\left[ n! \sum_{i=1}^n \frac{\frac{e^{py_i}-1}{y_i}}{\prod_{\substack{j=1 \\ j\neq i}}^n(y_i-y_j)}\right], \quad x\in \R^n. 
    \end{equation*}
\end{lemma}

Let us start by verifying Lemma \ref{SimplexLpSupport} for $n=2$ and $p=1$. The rest of the proof will follow by induction. 
\begin{claim}
\label{2dSimplexClaim}
For $x, y\in\R$, 
    \begin{equation*}
        \int_{\Delta_{2,+}} e^{xu+ yv}dudv =\frac{\frac{e^x-1}{x} - \frac{e^y - 1}{y}}{x-y}.  
    \end{equation*} 
\end{claim}
\begin{proof}
    Fix $x, y\in \R\setminus \{0\}$ with $x\neq y$. 
    Let us start with the integral: 
    \allowdisplaybreaks{
    \begin{align*}
        \int_{\Delta_{2,+}}e^{xu+ yv}du dv &= \int_0^1 \int_{0}^{1-u} e^{xu} e^{yv}dv du = \int_0^1 e^{xu} \left[ \frac{e^{yv}}{y}\right]_{v=0}^{1-u} du\\
        &= \frac1y \int_0^1 e^{xu} (e^y e^{-yu} - 1) du \\
        &= \frac1y \int_0^1 e^{u(x-y)} e^y - e^{xu} du \\
        &= \frac1y \left( \frac{e^{x-y}-1}{x-y} e^y - \frac{e^x-1}{x}\right) \\
        &= \frac{xe^x - x e^y - xe^x + x + ye^x - y}{xy(x-y)} \\
        &= \frac{y(e^x-1) - x(e^y - 1)}{xy(x-y)} \\
        &= \frac{\frac{e^x-1}{x} - \frac{e^y-1}{y}}{x-y}. 
    \end{align*}
    }
    Note, in addition, that this function is well-defined and smooth in $\R^2$. To see why, it is enough to use the Taylor expansion of the exponential function: 
    \begin{equation*}
        f(x,y)\defeq \frac{\frac{e^x-1}{x} - \frac{e^y-1}{y}}{x-y} = \frac{\sum_{n=1}^\infty \frac{x^{n-1} - y^{n-1}}{n!}}{x-y} = \sum_{n=2}^\infty \sum_{k=0}^{n-2} \frac{x^k y^{n-2-k}}{n!}, \quad (x,y)\in\R^2. 
    \end{equation*}
    Since both functions are continuous and agree on $\R^2\setminus \{x=0\}\cup \{y=0\}\cup \{x=y\}$, equality holds on $\R^2$.
\end{proof}

\begin{proof}[Proof of Lemma \ref{SimplexLpSupport}]
    It is enough to prove 
    \begin{equation}
        \int_{\Delta_{n,+}}e^{\langle x,y\rangle}dx = \sum_{i=1}^n \frac{\frac{e^{y_i}-1}{y_i}}{\prod_{\substack{j=1\\j\neq i}}^n (y_i-y_j)}
    \end{equation}
    The proof is by induction on $n$. The $n=2$ case is exactly Claim \ref{2dSimplexClaim}. For $n<2$, write $x= (\xi, x_n)\in \R^{n-1}\times \R$ and $y= (\eta, y_n)\in \R^{n-1}\times \R$, 
    \allowdisplaybreaks{
    \begin{align*}
        \int_{\Delta_{n,+}} e^{\langle x,y\rangle}dx &= \int_{x_n=0}^1 \int_{(1-x_n)\Delta_{n-1, +}}e^{\langle \xi, \eta\rangle + x_n y_n}d\xi dx_n \\
        &= \int_0^1 e^{x_n y_n} \int_{\Delta_{n-1,+}} e^{\langle (1-x_n)\xi', \eta\rangle} (1-x_n)^{n-1}d\xi' dx_n \\
        &= \int_0^1 e^{x_n y_n} (1-x_n)^{n-1} \sum_{i=1}^{n-1} \frac{\frac{e^{(1-x_n)y_i}-1}{(1-x_n)y_i}}{\prod_{\substack{j=1\\j\neq i}}^{n-1}((1-x_n)y_i- (1-x_n)y_j)}dx_n \\
        &= \int_0^1 e^{x_n y_n} \sum_{i=1}^{n-1}\frac{\frac{e^{(1-x_n)y_i}-1}{y_i}}{\prod_{\substack{j=1\\j\neq i}}^{n-1}(y_i- y_j)}dx_n \\
        &= \sum_{i=1}^{n-1}\frac{1}{y_i \prod_{\substack{j=1\\j\neq i}}^{n-1}(y_i - y_j)} \int_0^1 e^{x_ny_n} (e^{(1-x_n)y_i}-1) dx_n\\
        &= \sum_{i=1}^{n-1}\frac{1}{y_i \prod_{\substack{j=1\\j\neq i}}^{n-1}(y_i - y_j)} \left[ e^{y_i} \int_0^1 e^{x_n(y_n- y_i)}dx_n - \int_0^1 e^{x_n y_n}dx_n\right] \\
        &= \sum_{i=1}^{n-1}\frac{1}{y_i \prod_{\substack{j=1\\j\neq i}}^{n-1}(y_i - y_j)} \left[ \frac{e^{y_n}- e^{y_i}}{y_n  - y_i} - \frac{e^{y_n}-1}{y_n}\right] \\
        &= \sum_{i=1}^{n-1}\frac{1}{\prod_{\substack{j=1\\j\neq i}}^{n-1}(y_i - y_j)} \frac{\frac{e^{y_i}-1}{y_i}- \frac{e^{y_n}-1}{y_n}}{y_i- y_n} \\
        &=  \sum_{i=1}^{n-1}\frac{\frac{e^{y_i}-1}{y_i}}{\prod_{\substack{j=1\\j\neq i}}^{n}(y_i - y_j)}  - \frac{e^{y_n}-1}{y_n} \sum_{i=1}^{n-1}\frac{1}{\prod_{\substack{j=1\\ j\neq i}}^n (y_i -y_j)} \\
        &= \sum_{i=1}^{n-1}\frac{\frac{e^{y_i}-1}{y_i}}{\prod_{\substack{j=1\\j\neq i}}^{n}(y_i - y_j)}  + \frac{e^{y_n}-1}{y_n}\frac{1}{\prod_{\substack{j=1\\ j\neq n}}^n (y_n-y_j)}, 
    \end{align*} 
    }
    where the last equality follows from Claim \ref{SimplexAuxClaim2} below. 
\end{proof}

\begin{claim}
\label{SimplexAuxClaim2}
    For distinct $y_1, \ldots, y_n\in \R$, 
    \begin{equation*}
        \sum_{i=1}^{n-1} \frac{1}{\prod_{\substack{j=1\\ j\neq i}}^n (y_i -y_j)}= - \frac{1}{\prod_{\substack{j=1\\ j\neq n}}^n (y_i -y_j)}. 
    \end{equation*}
\end{claim}
\begin{proof}
    The claim is equivalent to showing: 
    \begin{equation*}
        \sum_{i=1}^{n} \frac{1}{\prod_{\substack{j=1\\ j\neq i}}^n (y_i -y_j)}= 0.
    \end{equation*}
    Consider the polynomial 
    \begin{equation*}
        P_{n-1}(x)\defeq \sum_{i=1}^n \frac{\prod_{\substack{j=1 \\ j\neq i}}^n (x - y_j)}{\prod_{\substack{j=1 \\ j\neq i}}^n (y_i - y_j)} =  \left[  \sum_{i=1}^{n} \frac{1}{\prod_{\substack{j=1\\ j\neq i}}^n (y_i -y_j)}\right] x^{n-1} + \text{ lower order terms}. 
    \end{equation*}
    This is such that $P_{n-1}(y_i)=1$ for all $i\in \{1, \ldots, n\}$, as all the terms of the sum vanish except one which evaluates to 1. Therefore, $P_{n-1}-1$ is a polynomial of degree $n-1$ that has $n$ distinct roots and must hence be constant $P_{n-1}\equiv 1$. In particular, the coefficient of the highest order term vanishes, hence the claim. 
\end{proof}

\section{\texorpdfstring{$L^p$-Mahler Integral}{Lp-Mahler Integral}}
\label{LpMahlerIntegralSection}

Following Section \ref{LegendreTransformSection}, we start in \S \ref{LpMahlerIntegralProperties} by exploring a few basic properties of $\M_p$ derived from the corresponding properties of the $L^p$-Legendre transform in \S\ref{LpLegendreProperties}. In \S\ref{LpMahlerIntegralExamples}, we use the calculations of \S\ref{LpLegendreExamples} to compute the Mahler integral of a few examples. Finally, in \S\ref{NakamuraTsujiConnection} we discuss the connection between $\M_p$ and the work of Nakamura--Tsuji.

\subsection{Basic properties}
\label{LpMahlerIntegralProperties}

\subsubsection{Convergence}
\begin{lemma}
    For $\phi\in\Cvx(\R^n)$, $\lim_{p\to \infty}\M_p(\phi)= \M(\phi)$. 
\end{lemma}
\begin{proof}
    By Lemma \ref{listLemma}(iii), $\{\phi^{*,p}\}_{p>0}$ is an increasing sequence. Therefore, by the Monotone Convergence Theorem \cite[\S 2.14]{Folland99} and Corollary \ref{LpConvergence}, 
    \begin{equation*}
        \lim_{p\to\infty} V(\phi^{*,p})= \lim_{p\to\infty}\int_{\R^n}e^{-\phi^{*,p}(y)}dy = \int_{\R^n}\lim_{p\to\infty}e^{-\phi^{*,p}(y)}dy = \int_{\R^n} e^{-\phi^*(y)}dy= V(\phi^*). 
    \end{equation*}
    Since $\M_p(\phi)= V(\phi)V(\phi^{*,p})$, the claim follows.
\end{proof}

\subsubsection{Inequalities}
A direct implication of Lemma \ref{listLemma}(iii) is that for $\phi\in\Cvx(\R^n)$, 
\begin{equation}
    \M_q(\phi)\leq\M_p(\phi), \quad p\leq q.  
\end{equation}
A reverse inequality holds under the extra assumption of the vanishing of the barycenter. 
\begin{lemma}
    Let $0<p < q \leq \infty$. For $\phi\in\Cvx(\R^n)$ with $b(\phi)=0$, 
    \begin{equation*}
        \M_q(\phi)\leq \left( \frac{p}{(1+p)^{1+\frac1p}}\right)^n \M_p(\phi). 
    \end{equation*}
\end{lemma}
\begin{proof}
    By Lemma \ref{ReverseIneq}, 
    \begin{equation*}
        \begin{aligned}
            V(\phi^{*,q})= \int_{\R^n}e^{-\phi^{*,q}(y)}dy\geq (1+p)^{-\frac{n}{p}}\int_{\R^n}e^{-\phi^{*,p}(\frac{p+1}{p}y)}dy = \left( \frac{p}{(1+p)^{1+\frac1p}}\right)^n V(\phi^{*,p}),
        \end{aligned}
    \end{equation*}
    by changing variables. Multiply both sides by $V(\phi)$ to get the claim. 
\end{proof}

\subsubsection{Tensoriality}
As an immediate result of Lemma \ref{LpLegendreTensor} and \eqref{tensorVolume} we obtain the following: 
\begin{corollary}
\label{MpTensoriality}
    Let $p\in (0,\infty)$. For $f:\R^n\to \R\cup\{\infty\}$ and $g:\R^m \to \R\cup\{\infty\}$ with $0< V(f), V(g)<\infty$, 
    \begin{equation*}
        \M_p(f\otimes g) = \M_p(f)\M_p(g). 
    \end{equation*}
\end{corollary}

\subsection{Examples} 
\label{LpMahlerIntegralExamples}
\subsubsection{\texorpdfstring{$L^1$-norm}{L1-norm}}
\begin{lemma}
    For $p\in (0,\infty)$, 
    $$
    \M_p(\|x\|_1)= \left( \frac{(p+1)^{1+\frac1p}}{p}\frac{2\sqrt{\pi}\, \Gamma(1+\frac1p)}{\Gamma(\frac32 + \frac1p)}\right)^n. 
    $$
\end{lemma}
\begin{proof}
    Denote by $\phi_1(x)\defeq \|x\|_1$. 
    By Lemma \ref{MpTensoriality}, it is enough to prove the claim for $n=1$. By Lemma \ref{LpLegendreL1norm}, 
    \begin{equation}
    \label{MpL1Eq1}
        \int_{\R}e^{-\phi_1^{*,p}(y)}dy = (p+1)^{1/p} \int_{-\frac{p+1}{p}}^{\frac{p+1}{p}} \left[1-\left( \frac{p}{p+1}y\right)^2\right]^{1/p}dy= \frac{(p+1)^{1+\frac1p}}{p} \int_{-1}^1 (1- z^2)^{\frac{1}{p}}dz.
    \end{equation}
    Recall that for the Beta function, 
    \begin{equation*}
        \mathrm{B}(x,y)\defeq \int_0^1 t^{x-1} (1-t)^{y-1}dt  = \frac{\Gamma(x)\Gamma(y)}{\Gamma(x+y)}, \quad x,y >0. 
    \end{equation*}
    Therefore, 
    \begin{equation}
    \label{MpL1Eq2}
        \int_{-1}^1 (1-z^2)^{\frac1p}dz = 2\int_0^1 (1-z^2)^{\frac1p}dz = \int_{0}^1 (1-w)^{\frac1p} w^{-\frac12} dw= \mathrm{B}\left( 1+\frac1p, \frac12\right) = \frac{\sqrt{\pi}\, \Gamma(1+\frac1p)}{\Gamma(\frac32 + \frac1p)}
    \end{equation}
    because $\Gamma(1/2)=\sqrt{\pi}$. The claim follows from \eqref{MpL1Eq1}--\eqref{MpL1Eq2} and \eqref{L1Volume}. 
\end{proof}

\begin{corollary}
    $\M_1(\|x\|_1) = (32/3)^n$. 
\end{corollary}

\subsubsection{Convex quadratic}
\begin{lemma}
\label{convexQuadratic}
    For $p\in (0,\infty)$, 
    \begin{equation*}
        \M_p(|x|^2/2) = \left( 2\pi \sqrt{\frac{(1+p)^{\frac{1+p}{p}}}{p}}\right)^n. 
    \end{equation*}
\end{lemma}
\begin{proof}
    By Lemma \ref{ConvexQuadraticLpLegendre} and \eqref{|x|^2/2Vol}, 
    \begin{equation*}
    \begin{aligned}
        \M_p(|x|^2/2) &= V(|x|^2/2) \int_{\R^n} e^{-\frac{p}{p+1}\frac{|y|^2}{2} + \frac{n}{2p}\log(1+p)}dy \\
        &= (2\pi)^{\frac{n}{2}} (1+ p)^{\frac{n}{2p}} \int_{\R^n}e^{-\frac{|\sqrt{\frac{p}{p+1}} \,y|^2}{2}} dy\\
        &= (2\pi)^{\frac{n}{2}} (1+p)^{\frac{n}{2p}} \int_{\R^n} e^{-\frac{|z|^2}{2}}\left(\frac{p+1}{p}\right)^{\frac{n}{2}}dz\\
        &= (2\pi)^n\frac{(1+p)^{\frac{n}{2}\frac{1+p}{p}}}{p^{\frac{n}{2}}}, 
    \end{aligned}
    \end{equation*}
    where the substitution $z=\sqrt{\frac{p}{p+1}}\, y$ was used. 
\end{proof}

\subsubsection{Functional simplex}
\begin{lemma}
\label{MpFuncSimplex}
    For $p\in (0,\infty)$, 
    \begin{equation*}
        \M_p(\phi_{\Delta_n}) = \left( e^{1+\frac{1}{p}}p^{\frac1p} \Gamma(1+ \frac1p)\right)^n.
    \end{equation*}
\end{lemma}
\begin{proof}
    Since $\phi_{\Delta_n}(y)= \phi_{\Delta_1}(y_1)+\ldots + \phi_{\Delta_1}(y_n)$, using Corollary \ref{MpTensoriality}, it is enough to prove the claim for $n=1$. 
    By Lemma \ref{functionalSimplexLpLegendre} and \eqref{functionalSimplexVolume}, 
    \begin{equation*}
        \M_p(\phi_{\Delta_1}) = V(\phi_{\Delta_1}) \int_{-\infty}^{\frac{1+p}{p}} e^{-1+ y +\frac1p \log(p + 1 -p y)} dy= e\int_{-\infty}^{\frac{1+p}{p}}e^{-1+y}(p+1 - py)^{\frac{1}{p}} dy. 
    \end{equation*}
    Setting $z= p+1 -py$ and then $t= z/p$, 
    \begin{equation*}
        \M_p(\phi_{\Delta_1}) = \int^{\infty}_0 e^{\frac{p+1-z}{p}} z^{\frac{1}{p}} \frac{dz}{p} = \frac{e^{1+\frac1p}}{p} \int_0^\infty e^{-\frac{z}{p}} z^{\frac1p}dz =  e^{1+\frac{1}{p}}p^{\frac1p}\int_{0}^\infty e^{-t} t^{\frac1p}dt = e^{1+\frac{1}{p}}p^{\frac1p} \Gamma(1+ \frac1p), 
    \end{equation*}
    hence the claim. 
\end{proof}

\subsubsection{The Euclidean ball}
Even though Lemma \ref{hpB2n} is very convenient for numerical experiments, it is not as convenient, at least for us, in the computation of $\M_p(B_2^n)$. In particular, we are not sure how one can go about computing the integral 
\begin{equation*}
    \int_0^\infty \frac{r^{n + \frac{n}{2p}-1}}{I_{n/2}(r)^{1/p}}dr. 
\end{equation*}

\begin{corollary}
\label{MpB2nCor}
    For $p\in (0,\infty)$, 
    \begin{equation*}
        \M_p(B_2^n) = \frac{n\pi^n}{2^{\frac{n}{2p}} p^n \Gamma(1+\frac{n}{2})^{2+ \frac1p}}\int_{0}^\infty \frac{r^{n + \frac{n}{2p} - 1}}{I_{n/2}(r)^{1/p}}dr= e^{o(n)} \left( \frac{2\pi e^{1+ \frac{1}{2p}}}{p n^{1+ \frac{1}{2p}}}\right)^n \int_0^\infty \frac{r^{n+\frac{n}{2p} - 1}}{I_{n/2}(r)^{1/p}}dr. 
    \end{equation*}
\end{corollary}
\begin{proof}
    By Lemma \ref{hpB2n}, 
    \begin{equation*}
        \begin{aligned}
            \M_p(B_2^n) &= |B_2^n| \frac{1}{2^{\frac{n}{2p}}\Gamma(1+ \frac{n}{2})^{\frac1p}} \int_{\R^n} \frac{|py|^{\frac{n}{2p}}}{I_{n/2}(p|y|)^{1/p}} dy.
        \end{aligned}
    \end{equation*}
    Using the change of variables $z = py, dz= p^n dy$, 
    \begin{equation*}
        \M_p(B_2^n) = |B_2^n| \frac{1}{2^{\frac{n}{2p}}\Gamma(1+ \frac{n}{2})^{\frac1p}} \int_{\R^n} \frac{|z|^{\frac{n}{2p}}}{I_{n/2}(|z|)^{1/p}}\frac{dz}{p^n}= \frac{n |B_2^n|^2}{2^{\frac{n}{2p}} p^n \Gamma(1+ \frac{n}{2})^{\frac1p}} \int_{0}^\infty \frac{r^{n + \frac{n}{2p} - 1}}{I_{n/2}(r)^{1/p}}dr. 
    \end{equation*}
    where polar coordinates and $|\partial B_2^n|= n |B_2^n|$ were used. 
    The first claim then follows from the formula for the volume of an $n$-ball, $|B_2^n| = \frac{\pi^{n/2}}{\Gamma(1+ \frac{n}{2})}$. 

    For the asymptotic behaviour, we use Stirling's formula $\Gamma(x)\sim \sqrt{2\pi x}\,(\frac{x}{e})^x$, to find 
    \begin{equation*}
        \Gamma(1+ \frac{n}{2}) = e^{o(n)} \left( \frac{n}{2e}\right)^{\frac{n}{2}}, 
    \end{equation*}
    hence $\Gamma(1+ \frac{n}{2})^{2+ \frac1p} = e^{o(n)} (\frac{n}{2e})^{n + \frac{n}{2p}}$, and 
    \begin{equation*}
        \frac{n\pi^n}{2^{\frac{n}{2p}}p^n \Gamma(1+ \frac{n}{2})^{2+\frac1p}} = e^{o(n)} \left( \frac{2\pi e^{1 + \frac{1}{2p}}}{p n^{1 + \frac{1}{2p}}}\right)^{n}, 
    \end{equation*}
    proving the second claim. 
\end{proof}

\subsubsection{The simplex}
We do not know how to explicitly compute $\M_p(\Delta_{n,+})$ for any finite $p$. Nonetheless, we conjecture that asymptotically, the $L^p$-Mahler volume of the simplex should equal the $L^p$-Mahler integral of the functional simplex (Lemma \ref{MpFuncSimplex}):
\begin{conjecture}
    For $p\in (0,\infty)$, 
    \begin{equation*}
        \M_p(\Delta_{n,0})= \left(e^{1+\frac1p} p^{\frac1p}\Gamma(1+\frac1p) \right)^n e^{o(n)}. 
    \end{equation*}
\end{conjecture}

\subsection{Connection to the work of Nakamura--Tsuji}
\label{NakamuraTsujiConnection}
Nakamura--Tsuji prove the following \cite[Corollary 1.6]{NakamuraTsuji24}:
Denote by $\gamma_n(x)\defeq (2\pi)^{-n/2}e^{-|x|^2/2}$ the standard Gaussian. 
\begin{theorem}[Nakamura--Tsuji]
\label{NTtheorem}
    Let $a\in (0,1)$.
    For a non-negative and even $F\in L^a(\R^n)$, 
    \begin{equation*}
        \frac{\|F\|_{L^a}}{\|\mathcal{L}F\|_{L^{\frac{a}{a-1}}}} \leq \frac{\|\gamma_n\|_{L^a}}{\|\mathcal{L}\gamma_n\|_{L^{\frac{a}{a-1}}}}. 
    \end{equation*}
\end{theorem}

Given Lemma \ref{DetropicalizedEquiv}, it becomes apparent that Theorem \ref{LpFuncSantalo} for even functions and Theorem \ref{NTtheorem} are equivalent. 
\begin{proof}[Proof of Lemma \ref{DetropicalizedEquiv}]
    Write $\phi\defeq -a\log F$. 
    First, 
    \begin{equation*}
        \|F\|_{L^a}= \left( \int_{\R^n}F(x)^{a}dx\right)^{1/a} = \left( \int_{\R^n}e^{-a\log F(x)}dx \right)^{1/a} = \left(\int_{\R^n} e^{-\phi(x)}dx\right)^{1/a}= V(\phi)^{1/a}. 
    \end{equation*}
    In addition, 
    \begin{equation*}
        (\mathcal{L}F)(y)\defeq \int_{\R^n} e^{\langle x,y\rangle} F(x) dx= \int_{\R^n} e^{\langle x,y\rangle + \log F(x)}dx = \int_{\R^n}e^{\langle x,y\rangle -\frac1a \phi(x)}dx. 
    \end{equation*}
    Since $p= \frac{1-a}{a}=\frac{1}{a}-1$, we have $\frac1a = p+1$, and hence 
    \begin{equation*}
        (\mathcal{L}F)(y) = \int_{\R^n}e^{\langle x,y\rangle - (p+1)\phi(x)}dx = V(\phi)\int_{\R^n} e^{p\langle x, y/p\rangle- (p+1)\phi(x)}\frac{dx}{V(\phi)}= V(\phi)e^{p\phi^{*,p}(y/p)}. 
    \end{equation*}
    Therefore, since $\frac{a}{a-1}= -\frac1p$, 
    \begin{equation*}
        \|\mathcal{L}F\|_{L^{\frac{a}{a-1}}}= \left( \int_{\R^n} (\mathcal{L}f)(y)^{\frac{a}{a-1}}dy\right)^{\frac{a-1}{a}} = \left(\int_{\R^n} V(\phi)^{-\frac1p}e^{-\phi^{*,p}(y/p)}dy\right)^{-p}= p^{-np}V(\phi)V(\phi^{*,p})^{-p}. 
    \end{equation*}
    Consequently, 
    \begin{equation*}
        \frac{\|F\|_{L^a}}{\|\mathcal{L}F\|_{L^{\frac{a}{a-1}}}}= p^{np} V(\phi)^{\frac{1}{a}-1} V(\phi^{*,p})^p = p^{np} V(\phi)^{p}V(\phi^{*,p})^p= p^{np} \M_p(\phi)^p
    \end{equation*}
    because $\frac{1}{a}-1=p$. 
\end{proof}

\section{\texorpdfstring{$L^p$-Santal\'o points}{Lp-Santal\'o points}}
\label{LpSantaloPointSection}
In this section, we generalize our result on the uniqueness and existence of $L^p$-Santal\'o points \cite[\S 4]{BMR23} from bodies to functions, by proving Proposition \ref{LpSantaloPoint}. 
The idea of proof is the same as in the case of convex bodies: $x\mapsto \M_p(T_x\phi)$ is a smooth strictly convex function when restricted to $\mathrm{int}\,\{\phi<\infty\}$ and blows-up on its complement. Consequently, it attains a unique minimum.

\subsection{Convexity and existence of a global minimum}
Let us review a fundamental property of lower-semicontinuous convex functions. A function $f:\R^n\to \R\cup\{\infty\}$ is lower-semi continuous when 
\begin{equation*}
    \liminf_{x\to x_0}f(x)\geq f(x_0), 
\end{equation*}
for all $x_0\in \R^n$. For an in-depth discussion on the existence of a minimum for a convex function, we refer the reader to \cite[\S 27]{Rock70}.
\begin{lemma}
\label{ConvMinimumEx}
    A proper lower semi-continuous convex function $\phi$ that satisfies 
    $
    \lim_{|x|\to \infty}\phi(x)=\infty
    $
    must attain a global minimum. 
\end{lemma}
\begin{proof}
    Recall we are only dealing with convex functions that are not identically $\infty$. Therefore 
    \begin{equation*}
        \beta\defeq \inf_{x\in\R^n}\phi(x)
    \end{equation*}
    is not equal to $\infty$. By the assumption that the limit of $\phi$ tends to $\infty$ as $|x|$ tends to infinity, there exists $R>0$ such that $\phi(x)\geq 2\beta$ for all $|x|>R$. In particular, there exist $\{x_m\}_{m\geq 1}\subset RB_2^n$ such that $\lim_{m\to\infty}\phi(x_m)= \beta$. Since $RB_2^n$ is closed, there is $x_0$ and a subsequence $\{x_{k_m}\}_{m\geq 1}\subset \{x_m\}_{m\geq1}$ with $x_{k_m}\to x_0$. By lower semi-continuity, 
    \begin{equation*}
        \phi(x_0)\leq \liminf_{m\to\infty}\phi(x_{k_m})= \beta. 
    \end{equation*}
    In particular, since $\beta\defeq \inf_{\R^n}\phi$, it must be $\phi(x_0)=\beta$, hence the minimum is attained. Note, in particular, that $\beta>-\infty$ as $\phi$ does not take the $-\infty$ value. 
\end{proof}

\begin{remark}
    The exponential function $\phi(x)\defeq e^x, x\in \R^n$, is smooth and bounded from below but does not attain a global minimum. Consequently, Lemma \ref{ConvMinimumEx} does not hold without the assumption $\lim_{|x|\to\infty} \phi(x)=\infty$ (here: $\lim_{x\to -\infty}e^{x}=0$). 
\end{remark}

\begin{remark}
    Lemma \ref{ConvMinimumEx} does not hold without the lower semi-continuity assumption. For example, 
    \begin{equation*}
        \phi(x)\defeq \begin{dcases}
            \infty, & x< 0\\
            1,  & x = 0, \\
            x^2, & x>0, 
        \end{dcases}
    \end{equation*}
    is convex but does not attain a minimum. It is enough to demonstrate $\phi$ is bounded from below; the rest of the proof will be taken care of by lower semi-continuity. 
\end{remark}

\begin{corollary}
\label{GlobalMinEx}
    Any $\phi\in\Cvx(\R^n)$ attains a global minimum. 
\end{corollary}
\begin{proof}
    Let $\phi\in\Cvx(\R^n)$. By Lemma \ref{ConvexLowerBoundLemma}, there exist $a>0, b\in \R$ such that $\phi(x)\geq a|x|+b$. In particular, $\lim_{|x|\to\infty}\phi(x)=\infty$, hence by Lemma \ref{ConvMinimumEx}, $\phi$ attains a global minimum. 
\end{proof}

\subsection{Finiteness properties}
In the case of convex bodies $\M_p(K-x)$ is finite if and only if $x\in\mathrm{int}\,K$ \cite[Lemma 4.2]{BMR23}. This generalizes to functions. 
\begin{lemma}
\label{FinitenessLemma}
    Let $p\in (0,\infty)$. For $\phi\in \Cvx(\R^n)$, $\M_p(T_x\phi)<\infty$ if and only if $x\in\mathrm{int}\{\phi<\infty\}$. 
\end{lemma}

The proof of Lemma \ref{FinitenessLemma} will follow from the following two claims, generalizing the corresponding geometric claims \cite[Claims 4.8 and 4.9]{BMR23}:
\begin{claim}
\label{FinitenessClaim1}
    Let $p\in (0,\infty)$ and $\phi\in\Cvx(\R^n)$ with $0\in\mathrm{int}\,\{\phi<\infty\}$. There exists $r>0$ and $M>0$ such that  
    \begin{equation*}
        \M_p(\phi)\leq e^{\frac{p+1}{p}M}\frac{V(\phi)^{1+\frac1p}}{(2r)^{n+\frac{n}{p}}} \M_p([-1,1]^n). 
    \end{equation*}
    In particular, $\M_p(\phi)<\infty$. 
\end{claim}
\begin{proof}
    Since $0\in\mathrm{int}\,\{\phi<\infty\}$, there exists $r>0$ such that $[-r,r]^n\subset \mathrm{int}\,\{\phi<\infty\}$. When restricted to $\mathrm{int}\,\{\phi<\infty\}$, $\phi$ is continuous \cite[Theorem 10.1]{Rock70}. In particular, it is bounded on $[-r,r]^n$, i.e., there exists $M>0$ such that $\phi(x)\leq M$ for all $x\in [-r,r]^n$. Using this, for $y\in \R^n$, 
    \begin{equation*}
        \begin{aligned}
            \phi^{*,p}(y)& =\frac1p \log\int_{\R^n}e^{p\langle x,y\rangle - (p+1)\phi(x)}\frac{dx}{V(\phi)}\geq \frac1p \log\int_{[-r,r]^n} e^{p\langle x, y\rangle - (p+1)\phi(x)}\frac{dx}{V(\phi)} \\
            &\geq \frac1p \log\int_{[-r,r]^n} e^{p\langle x, y\rangle - (p+1)M}\frac{dx}{V(\phi)}= \frac1p \log\left[e^{-(p+1)M}\int_{[-r,r]^n}e^{p\langle x,y\rangle}\frac{dx}{|[-r,r]^n|}\frac{|[-r,r]^n|}{V(\phi)}\right]\\ 
            &=  -\frac{p+1}{p}M + h_{p, [-r,r]^n}(y)+ \frac1p\log \frac{(2r)^n}{V(\phi)}. 
        \end{aligned}
    \end{equation*}
    Consequently, 
    \begin{equation*}
        \begin{aligned}
            \M_p(\phi)= V(\phi)\int_{\R^n}e^{-\phi^{*,p}(y)}dy \leq e^{\frac{p+1}{p}M} \frac{V(\phi)^{1+\frac1p}}{(2r)^{\frac{n}{p}}}\int_{\R^n}e^{-h_{p,[-r,r]^n}(y)}dy.  
        \end{aligned}
    \end{equation*}
    Now, $\int_{\R^n}e^{-h_{p,[-r,r]^n}(y)}dy = \M_p([-r,r]^n)/|[-r,r]^n|= \M_p([-r,r]^n)/(2r)^n = \M_p([-1,1]^n)/(2r)^n$, by the linear invariance of $\M_p$ \cite[Lemma 4.7]{BMR23}. 
\end{proof}

\begin{claim}
\label{FinitenessClaim2}
    Let $p\in (0,\infty]$. For $\phi\in\Cvx(\R^n)$ with $0\not\in\mathrm{int}\,\{\phi<\infty\}$, $\M_p(\phi)=\infty$.  
\end{claim}
\begin{proof}
    By Corollary \ref{GlobalMinEx}, $\phi$ attains a global minimum. Replacing $\phi$ by $\phi-\inf_{\R^n}\phi$ leaves $\M_p$ invariant, hence we may assume $\phi\geq 0$. 
    
    By convexity of $\{\phi<\infty\}$, since $0\not\in \mathrm{int}\{\phi<\infty\}$, there exists $u\in\partial B_2^n$ and hyperplane through the origin 
    \begin{equation*}
        u^\perp\defeq \{x\in \R^n: \langle x,u\rangle =0\}
    \end{equation*}
    such that $\{\phi<\infty\}\subset  \{x\in \R^n: \langle x, u\rangle \geq 0\}$. In particular, $\langle x, -u\rangle \leq 0$ for all $x\in \{\phi<\infty\}$, thus 
    \begin{equation*}
        c\defeq e^{p\phi^{*,p}(-u)}= \int_{\{\phi<\infty\}}e^{p\langle x,-u\rangle -(p+1)\phi(x)}\frac{dx}{V(\phi)}\leq \int_{\{\phi<\infty\}}e^{p\langle x, -u\rangle}\frac{e^{-\phi(x)} dx}{V(\phi)} < 1
    \end{equation*}
    because $\phi\geq 0$. It is strictly less than one since $\{\phi<\infty\}$ is $n$-dimensional, hence not completely contained in $u^\perp$. By continuity, there exists an open neighborhood $U$ of $-u$ in $\partial B_2^n$ such that 
    \begin{equation*}
        e^{p\phi^{*,p}(rv)}\leq e^{p\phi^{*,p}(v)}\leq \int_{\{\phi<\infty\}} e^{p\langle x, v\rangle}\frac{e^{-\phi(x)} dx}{V(\phi)}< \frac{1+c}{2}, \quad \text{for all } v\in U, r\geq 1. 
    \end{equation*}
    Consequently, in polar coordinates
    \begin{equation*}
        \M_p(\phi)\geq V(\phi)\int_{\partial V}\int_{r=1}^\infty e^{-\phi^{*,p}(rv)}r^{n-1}drdv\geq V(\phi) \left(\frac{1+c}{2} \right)^{-1/p}\int_V \int_{1}^\infty r^{n-1}dr dv= \infty. 
    \end{equation*}
\end{proof}

\begin{proof}[Proof of Lemma \ref{FinitenessLemma}]
    It is enough to note that $x\in\mathrm{int}\,\{\phi<\infty\}$ if and only if $0\in\mathrm{int}\,\{T_x\phi<\infty\}$. By Claims \ref{FinitenessClaim1} and \ref{FinitenessClaim2}, it follows that $\M_p(T_x\phi)<\infty$ if and only if $x\in\mathrm{int}\,\{\phi<\infty\}$.  
\end{proof}

The next corollary holds trivially in case $\partial\{\phi<\infty\}= \emptyset$. 
\begin{corollary}
\label{MpBoundaryCor}
    Let $p\in (0,\infty]$. For $\phi\in\Cvx(\R^n)$ and $x_0\in \partial\{\phi<\infty\}$, 
    \begin{equation*}
        \lim_{x\to  x_0} \M_p(T_x\phi)= \infty. 
    \end{equation*}
\end{corollary}
\begin{proof}
     By Fatou's Lemma \cite[\S 2.18]{Folland99} and Lemma \ref{FinitenessLemma}, 
     \begin{equation*}
         \liminf_{x\to x_0}\M_p(T_x\phi)\geq \M_p(\liminf_{x\to x_0} T_x\phi)= \M_p(T_{x_0}\phi) = \infty. 
     \end{equation*}
     Therefore, the limit exists and equals $\infty$. 
\end{proof}

\subsection{Smoothness under translation and strict convexity}
\begin{lemma}
\label{SmoothnessLemma}
    Let $p\in (0,\infty]$. For $\phi\in\Cvx(\R^n)$, the map $x\mapsto \M_p(T_x\phi)$, when restricted to $\mathrm{int}\,\{\phi<\infty\}$, is smooth and strictly convex. Furthermore, $D_x\M_p(T_x\phi) = \M_p(T_x\phi) b((T_x\phi)^{*,p})$. 
\end{lemma}
\begin{proof}
    Denote by $e_1, \ldots, e_n$ the standard basis of $\R^n$. 
    Fix $x_0\in \mathrm{int}\,\{\phi<\infty\}$ and let $r>0$ such that $x_0 + 2rB_2^n \subset \mathrm{int}\,\{\phi<\infty\}$. For $0<\e<r$, using Lemma \ref{listLemma}(i),
    \begin{equation}
    \label{MpQuotientEq1}
        \begin{aligned}
            \frac{\M_p(T_{x_0+ \e e_i}\phi)-\M_p(T_{x_0}\phi)}{\e}
            &= V(\phi) \int_{\R^n}\frac{e^{\e y_i}-1}{\e} e^{-(T_{x_0}\phi)^{*,p}(y)}dy,  
        \end{aligned}
    \end{equation}
    because $(T_{x_0 + \e e_i}\phi)^{*,p}(y)= (T_{x_0}\phi)^{*,p}(y) - \langle \e e_i, y\rangle= (T_{x_0}\phi)^{*,p}(y)- \e y_i$. The goal is to apply the Dominated Convergence Theorem \cite[\S 2.24]{Folland99} to commute the integral and the limit. To use this theorem, we first need a uniform upper bound on the quotient \eqref{MpQuotientEq1} that is integrable and independent of $\e$. 
    Using the power series expansion of the exponential function:
    \begin{equation}
        \label{expBound}
        \left|\frac{e^{\e y_i}-1}{\e}\right| = \left| \sum_{m=1}^\infty\frac{\e^{m-1}}{m!} y_i^m\right|\leq \sum_{i=1}^\infty \frac{r^{m-1}}{m!} |y_i|^m= \frac{e^{r|y_i|}-1}{r} \leq \frac{e^{r|y_i|}}{r}. 
    \end{equation}
    By \eqref{MpQuotientEq1} and \eqref{expBound}, 
    \begin{equation*}
        \begin{aligned}
            \left|\frac{\M_p(T_{x_0+ \e e_i}\phi)-\M_p(T_{x_0}\phi)}{\e}\right| 
            &\leq V(\phi)\int_{\R^n} \frac{e^{r|y_i|}}{r}e^{-(T_{x_0}\phi)^{*,p}(y)}dy \\
            &= \frac{V(\phi)}{r} \int_{\R^{n-1}}\left( \int_0^\infty e^{ry_i} + \int_{-\infty}^0 e^{-ry_i}\right) e^{-(T_{x_0}\phi)^{*,p}(y)}dy\\
            &\leq \frac{V(\phi)}{r} \int_{\R^n}e^{ry_i}e^{-(T_{x_0}\phi)^{*,p}(y)} dy + \frac{V(\phi)}{r} \int_{\R^n} e^{-ry_i}e^{-(T_{x_0}\phi)^{*,p}(y)} dy\\
            &= \frac{V(\phi)}{r}\int_{\R^n}e^{-(T_{x_0+ r e_i}\phi)^{*,p}(y)}dy + \frac{V(\phi)}{r} \int_{\R^n} e^{-(T_{x_0 - re_i}\phi)^{*,p}(y)}dy \\
            &= \frac{\M_p(T_{x_0 + r e_i}\phi)+ \M_p(T_{x_0- re_i})}{r}, 
        \end{aligned}
    \end{equation*}
    which is finite by Lemma \ref{FinitenessLemma} because both $x_0+ re_i$ and $x_0- re_i$ lie in $\mathrm{int}\,\{\phi<\infty\}$. Therefore, the Dominated Convergence Theorem \cite[\S 2.24]{Folland99} applied to \eqref{MpQuotientEq1} gives 
    \begin{equation*}
        \begin{aligned}
            \lim_{\e\to 0}\frac{\M_p(T_{x_0+ \e e_i}\phi)- \M_p(T_{x_0}\phi)}{\e} 
            &= \lim_{\e\to 0} V(\phi)\int_{\R^n}\frac{e^{\e y_i}-1}{\e} e^{-(T_{x_0}\phi)^{*,p}(y)} dy\\
            &= V(\phi)\int_{\R^n}\lim_{\e\to 0}\frac{e^{\e y_i}-1}{\e}e^{-(T_{x_0}\phi)^{*,p}(y)} dy \\
            &= V(\phi)\int_{\R^n}y_i e^{-(T_{x_0}\phi)^{*,p}(y)} dy. 
        \end{aligned}
    \end{equation*}
    By multiplying and dividing the right-hand side by $V((T_{x_0}\phi)^{*,p})$, 
    \begin{equation*}
    \begin{aligned}
        \frac{\partial \M_p(T_x\phi)}{\partial x_i}(x_0)= V(\phi)V((T_{x_0}\phi)^{*,p})\int_{\R^n} y_i e^{-(T_{x_0}\phi)^{*,p}(y)}\frac{dy}{V((T_{x_0}\phi)^{*,p})} = \M_p(T_{x_0}\phi) b((T_{x_0}\phi)^{*,p})_i. 
    \end{aligned}
    \end{equation*}

    The existence of higher-order derivatives can be shown similarly. Specifically, 
    \begin{equation*}
        \frac{\partial^2 \M_p(T_x\phi)}{\partial x_i\partial x_j}(x_0)= V(\phi)\int_{\R^n}y_i y_j e^{-(T_{x_0}\phi)^{*,p}(y)} dy. 
    \end{equation*}
    Therefore, for any $v\in \R^n$, 
    \begin{equation*}
        v^T D^2_x\M_p(T_x\phi)v= V(\phi)\sum_{i,j=1}^n \int_{\R^n}y_iy_jv_iv_j e^{-(T_{x}\phi)^{*,p}(y)}dy =V(\phi)\int_{\R^n}\langle y,v\rangle^2 e^{-(T_{x}\phi)^{*,p}(y)}dy \geq 0, 
    \end{equation*}
    with equality if and only if $v=0$, demonstrating strict convexity. 
\end{proof}

\subsection{Existence of Santal\'o points}
In view of Lemma \ref{ConvMinimumEx}, in order to finish the proof of Proposition \ref{LpFunSantaloPoint} we require the following: 
\begin{lemma}
    Let $p\in (0,\infty]$. For $\phi\in\Cvx(\R^n)$, 
    $
        \lim_{|x|\to \infty} \M_p(T_x\phi)= \infty. 
    $
\end{lemma}
\begin{proof}
    Let us first deal with the radial limits: Fix $u\in \partial B_2^n$ and consider $\lim_{t\to\infty}\M_p(T_{tu}\phi)$. Restricting the integral to the cone where 
    $\langle y, u\rangle \geq 1$: 
    \begin{equation*}
        \int_{\R^n}e^{-(T_{tu}\phi)^{*,p}(y)}dy = \int_{\R^n} e^{-\phi^{*,p}(y)+ \langle tu, y\rangle}dy\geq \int_{\{y: \langle y,u\rangle \geq 1\}} e^{-\phi^{*,p}(y)+ t}dy, 
    \end{equation*}
    the first equality follows from Lemma \ref{listLemma}(i). 
    The Monotone Convergence Theorem \cite[\S 2.14]{Folland99} implies 
    \begin{equation*}
        \lim_{t\to\infty} \int_{\{y: \langle y,u\rangle\geq 1\}} e^{-\phi^{*,p}(y)+t}dy = \int_{\{y: \langle y,u\rangle \geq 1\}}\lim_{t\to\infty} e^{-\phi^{*,p}(y)+t} dy = \infty. 
    \end{equation*}

     In the general case one may argue by contradiction: suppose there exists a sequence $\{x_m\}_{m\geq 1}$ with $|x_m|\to \infty$ but $|\M_p(T_{x_m}\phi)|$ does not diverge to infinity. Then there must exist be a subsequence $\{x_{k_m}\}_{m\geq 1}\subset \{x_m\}_{m\geq 1}$ and a constant $M>0$ such that 
     \begin{equation}
     \label{boundedHypo}
         \M_p(T_{x_{k_m}}\phi)\leq M, \quad \text{ for all } m\in\mathbb{N}. 
     \end{equation} 
     Considering the sequence of directions $\{x_{k_m}/|x_{k_m}|\}_{m\geq 1}\subset \partial B_2^n$, the compactness of $\partial B_2^n$ implies the existence of a further subsequence $\{x_{l_{k_m}}\}_{m\geq 1}\subset \{x_{k_m}\}_{m\geq 1}$ with $x_{l_{k_m}}/|x_{l_{k_m}}|$ converging to a unit vector $u\in\partial B_2^n$. For sufficiently large $m$, if $y$ is such that $\langle y, u\rangle\geq 1$, then $\langle y, x_{l_{k_m}}/ |x_{l_{k_m}}|\rangle\geq 1/2$, or equivalently, $\langle y, x_{l_{k_m}}\rangle\geq |x_{l_{k_m}}|/2$. Thus, for large enough $m$, using Lemma \ref{listLemma}(i), 
     \begin{equation*}
         \int_{\R^n}e^{-(T_{x_{l_{k_m}}} \phi)^{*,p}(y)}dy = \int_{\R^n} e^{-\phi^{*,p}(y) + \langle y, x_{l_{k_m}}\rangle}dy \geq \int_{\{y: \langle y,u\rangle\geq 1\}}e^{-\phi^{*,p}(y) + |x_{l_{k_m}}|/2} dy.
     \end{equation*}
     By further restricting to a subsequence, we may assume $\{|x_{l_{k_m}}|\}_{m\geq 1}$ is increasing, hence the Monotone Convergence Theorem \cite[\S 2.14]{Folland99} gives $\lim_{m\to\infty}\M_p(T_{x_{l_{l_m}}}\phi)=\infty$, contradicting \eqref{boundedHypo}. 
\end{proof}

\begin{proof}[Proof of Proposition \ref{LpFunSantaloPoint}]
    The lower semi-continuity of $x\mapsto \M_p(T_x\phi)$ follows from Lemmas \ref{FinitenessLemma}, \ref{SmoothnessLemma}, and Corollary \ref{MpBoundaryCor}. The existence of a minimum $s_p(\phi)$ follows from the convexity of the map (Lemma \ref{SmoothnessLemma}) and Lemma \ref{GlobalMinEx}. Uniqueness of $s_p(\phi)$ follows from strict convexity. By Lemma \ref{FinitenessLemma}, since $D_x\M_p(T_{s_p(\phi)}\phi)=0$, we conclude that $b((T_{s_p(\phi)} \phi)^{*,p}=0$. 
\end{proof}

\section{Asymptotics}
\label{Asymptotics}

One potential approach to establishing Theorem \ref{LpFuncSantalo} is to leverage the corresponding geometric result (Theorem \ref{LpSantaloThm}), following the methods of Artstein--Klartag--Milman \cite{AKM04} and Fradelizi--Meyer \cite{FM08a}. The approach requires two main components. First, we need to approximate the $L^p$-Mahler integral of the function by the $L^p$-Mahler volume of convex bodies, potentially in a much higher dimension. This approximation should also ensure that the $L^p$-Santal\'o points of the approximating bodies converge to the $L^p$-Santal\'o point of the underlying function. Second, obtain an upper bound on the Mahler integral by Theorem \ref{LpSantaloThm} by sending the auxiliary dimension to infinity. This step requires knowing the dimensional asymptotics of $\M_p(B_2^n)$, which should match the asymptotic behavior corresponding maximizer of the functional setting. In our case, we would need the following: 
\begin{conjecture}
\label{MpB2nConj}
    For $p\in (0,\infty)$, 
    \begin{equation*}
        \M_p(B_2^n)= \M_p(|x|^2/2)e^{o(n)}= \left( 2\pi \sqrt{\frac{(1+p)^{\frac{1+p}{p}}}{p}}\right)^n e^{o(n)}. 
    \end{equation*}
\end{conjecture}
By Corollary \ref{MpB2nCor}, we have an explicit formula for $\M_p(B_2^n)$. However, this formula involves the integral of a Bessel function, which is challenging to compute. Specifically, in view of Corollary \ref{MpB2nCor}, Conjecture \ref{MpB2nConj} becomes equivalent to the following: 
\begin{conjecture}
\label{BesselIntegralConj}
For $p\in (0,\infty)$, 
        \begin{equation}
        \label{GOAL}
        \int_0^\infty \frac{r^{n+\frac{n}{2p}-1}}{I_{n/2}(r)^{\frac1p}}dr = \left(\frac{n^{1+ \frac{1}{2p}}}{e^{1+ \frac{1}{2p}}} \sqrt{p(1+p)^{1+ \frac{1}{p}}}\right)^n e^{o(n)}.
    \end{equation}
\end{conjecture}
We prove Conjectures \ref{MpB2nConj} and \ref{BesselIntegralConj} for $p=1$. It is possibly useful to mention we found the following work useful in our calculations: \cite{Gaunt14, Gaunt16, Gaunt18, Gaunt19, Gaunt20}. Also, an interesting book for asymptotic methods for integrals \cite{Temme15}.

\subsection{Reformulation: From first kind to second kind}
Denote by 
\begin{equation*}
    K_\nu(x) 
\end{equation*}
the modified Bessel function of the second kind. 
Our goal is to study the integral \eqref{GOAL}. However, having the Bessel function in the denominator is an integral that we could not trace in the literature. To overcome this problem, we use an inequality of Gaunt \cite[Lemma 3]{Gaunt16}: 
\begin{lemma}
\label{GauntLemma}
    For $\nu\geq -1/2$ and $r\geq 0$, $\frac12 < r K_{\nu+1}(r)I_\nu(r)\leq 1.$
\end{lemma}
Using Lemma \ref{GauntLemma} we may replace $I_{n/2}(r)^{-\frac{1}{p}}$ in \eqref{GOAL} by $r^{\frac1p}K_{\frac{n}{2}+1}(r)^{\frac1p}$: 
\begin{corollary}
    \label{GauntCorollary}
    For $p\in (0,\infty)$, 
    \begin{equation*}
        \int_0^\infty \frac{r^{n+\frac{n}{2p}-1}}{I_{n/2}(r)^{\frac1p}}dr= e^{o(n)}\int_{0}^\infty r^{n+ \frac{3n}{2p}-1} K_{\frac{n}{2}+1}(r)^{\frac{1}{p}}dr.  
    \end{equation*}
\end{corollary}

\subsection{\texorpdfstring{The $p=1$ case}{The p=1 case}}
When $p=1$ we can use the well-known formula \cite[(A.17)]{Gaunt14}, 
\begin{equation}
\label{KnuInt}
    \int_{\R}e^{\beta t} |t|^\nu K_\nu(|t|)dt = \frac{\sqrt{\pi}\,\Gamma(\nu+1/2)2^\nu}{(1-\beta^2)^{\nu + \frac12}}, \quad \nu>-1/2, -1<\beta<1.  
\end{equation}
\begin{lemma}
\label{KintAsympt1}
    Let $n\in\mathbb{N}$ be odd. For $m\in\mathbb{N}$, 
    \begin{equation*}
        \int_{0}^\infty t^{2m+ \frac{n}{2}+1} K_{\frac{n}{2}+1}(t)dt = \sqrt{\pi}\,\Gamma\Big(\frac{n+1}{2}+1\Big) 2^{\frac{n}{2}}\frac{(m+\frac{n+1}{2})! (2m)!}{m!(\frac{n+1}{2})!}. 
    \end{equation*}
\end{lemma}
\begin{proof}
    The idea is to expand both sides of \eqref{KnuInt} as functions of $\beta$ and then compare coefficients. For the left-hand side, for any $\nu>-1/2$: 
    \begin{equation}
    \label{betaExp1}
        \int_{\R}e^{\beta t}|t|^\nu K_{\nu}(|t|)dt = \sum_{m=0}^\infty \frac{\beta^m}{m!}\int_{\R}t^m |t|^\nu K_{\nu}(|t|)dt = \sum_{m=0}^\infty \frac{2\beta^{2m}}{(2m)!}\int_0^\infty t^{2m+\nu} K_\nu(t)dt, 
    \end{equation}
    because 
    \begin{equation*}
        \int_{\R} t^m|t|^\nu K_\nu(|t|)dt = \begin{dcases}
            2\int_0^\infty t^{m+\nu} K_{\nu}(t), &\text{ if } m \text{ is even},\\
            0, & \text{ if } m \text{ is odd}. 
        \end{dcases}
    \end{equation*}
    
    For the right-hand side of \eqref{KnuInt}, starting with    
    \begin{equation*}
        \frac{1}{1-\beta} = \sum_{m=0}^\infty \beta^m, 
    \end{equation*}
    and differentiating $k$ times:
    \begin{equation*}
        \frac{k!}{(1-\beta)^{k+1}} = \sum_{m=k}^\infty \frac{m!}{(m-k)!} \beta^{m-k} = \sum_{m=0}^\infty \frac{(m+k)!}{m!}\beta^m. 
    \end{equation*}
    Since $n$ was taken odd, $k=\frac{n+1}{2}$ is an integer. Therefore, 
    \begin{equation}
    \label{betaExp2}
        \frac{1}{(1-\beta^2)^{\frac{n+1}{2}+1}}= \sum_{m=0}^\infty \frac{(m+ \frac{n+1}{2})!}{m! (\frac{n+1}{2})!}\beta^{2m}.
    \end{equation}
    By \eqref{betaExp1} for $\nu= \frac{n}{2}+1$, \eqref{KnuInt} and \eqref{betaExp2}, 
    \begin{equation*}
        \sum_{m=0}^\infty \frac{2\beta^{2m}}{(2m)!}\int_{0}^\infty t^{2m+\frac{n}{2}+1}K_{\frac{n}{2}+1}(t)dt = \sqrt{\pi}\,\Gamma\Big(\frac{n+1}{2}+1\Big) 2^{\frac{n}{2}+1}\sum_{m=0}^\infty \frac{(m+\frac{n+1}{2})!}{m! (\frac{n+1}{2})!}\beta^{2m}. 
    \end{equation*}
    Comparing the coefficients, we find 
    \begin{equation*}
        \frac{2}{(2m)!}\int_0^\infty t^{2m+\frac{n}{2}+1}K_{\frac{n}{2}+1}(t)dt = \sqrt{\pi}\,\Gamma\Big(\frac{n+1}{2}+1\Big) 2^{\frac{n}{2}+1}\frac{(m+\frac{n+1}{2})!}{m!(\frac{n+1}{2})!}, 
    \end{equation*}
    hence the claim. 
\end{proof}

Consequently, for $n\in \mathbb{N}$ odd and $m=\frac{n-1}{2}$: 
\begin{corollary}
\label{p1Calc}
    For $n\in\mathbb{N}$ odd, 
    \begin{equation*}
        \int_{0}^\infty t^{n+ \frac{n}{2}} K_{\frac{n}{2}+1}(t)dt = \sqrt{\pi}\,\Gamma\Big( \frac{n+1}{2}+1\Big) 2^{\frac{n}{2}}\frac{n! (n-1)!}{(\frac{n-1}{2})!(\frac{n+1}{2})!} = e^{o(n)}\left( \frac{2n^{3/2}}{e^{3/2}}\right)^n
    \end{equation*}
\end{corollary}
\begin{proof}
    The first equation is a direct consequence of Lemma \ref{KintAsympt1} for $m= \frac{n-1}{2}$. For the asymptotics, using Stirling's formula: 
    \begin{equation*}
        \sqrt{\pi}\,\Gamma\Big(\frac{n+1}{2}+1\Big) 2^{\frac{n}{2}} \frac{(n+1)!(n-1)!}{(\frac{n-1}{2})!(\frac{n+1}{2})!}= e^{o(n)}\Gamma\Big(\frac{n}{2}\Big)2^{\frac{n}{2}}\frac{\Gamma(n)^2}{\Gamma(\frac{n}{2})^2}= e^{o(n)}\left( \frac{2n^{3/2}}{e^{3/2}}\right)^n. 
    \end{equation*}
\end{proof}

\begin{proof}[Proof of Conjecture \ref{MpB2nConj} for $p=1$]
    By Corollaries \ref{MpB2nCor}, \ref{p1Calc}, and \ref{GauntCorollary},   
    \begin{equation*}
        \M_1(B_2^n)= e^{o(n)} \left( \frac{2\pi e^{\frac{3}{2}}}{n^{\frac{3}{2}}}\right)^n \int_{0}^\infty r^{n+ \frac{3n}{2}-1} K_{\frac{n}{2}+1}(r) dr= e^{o(n)} \left( \frac{2\pi e^{3/2}}{n^{3/2}}\right)^n \left( \frac{2n^{3/2}}{e^{3/2}}\right)^n= e^{o(n)}(4\pi)^n,
    \end{equation*}
    where $(4\pi)^n$ is exactly equal to $\M_1(|x|^2/2)$ by Lemma \ref{convexQuadratic}. 
\end{proof}

\section{Evolution under the Fokker--Planck heat flow}
\label{FokkerPlanckSection}
In this section, we turn to the proof of Theorem \ref{LpFuncSantalo} via the Fokker--Planck heat flow. The parabolic equation \eqref{FPEq1} is formulated for log-concave functions. Since we prefer to instead work with convex functions, we arrive at the following parabolic PDE: 
\begin{equation}
\label{FPEq2}
    \begin{dcases}
        (\partial_t\phi)(t,x) = (\Delta\phi)(t,x) - |(D\phi)(t,x)|^2 + \langle x, (D\phi)(t,x)\rangle - n,  & (t, x)\in [0,\infty)\times \R^n, \\
        \phi(0, x) = \phi_0(x), & x\in \R^n. 
    \end{dcases}
\end{equation}
For a quick proof, let $f= e^{-\phi}$ be a log-concave function that satisfies \eqref{FPEq1}. Since $\partial_i f = -(\partial_i \phi) e^{-\phi}$, the gradient $Df= -e^{-\phi} D\phi$. In addition, $\partial^2_{ij}f = -(\partial^2_{ij}\phi) e^{-\phi}+ (\partial_i\phi)(\partial_j\phi)e^{-\phi}$, hence $\Delta f= \sum_{i=1}^n \partial^2_{ii}f = \sum_{i=1}^n -\partial^2_{ii}\phi\, e^{-\phi} + (\partial_i\phi)^2 e^{-\phi}= - (\Delta\phi)e^{-\phi} + |D\phi|^2 e^{-\phi}$. Consequently, $f$ satisfies \eqref{FPEq1} if and only if 
\begin{equation*}
    -(\partial_t\phi) e^{-\phi} = \partial_t f= \Delta f + \langle x, Df\rangle + n f= - (\Delta \phi)e^{-\phi} + |D\phi|^2 e^{-\phi}-\langle x, D\phi\rangle e^{-\phi}+ n e^{-\phi}. 
\end{equation*}
Cancelling out $e^{-\phi}$ shows that $\phi$ satisfies \eqref{FPEq2}. We have shown the following: 

\begin{lemma}
\label{FPequivalence}
    $\phi$ satisfies \eqref{FPEq2} if and only if $e^{-\phi}$ satisfies \eqref{FPEq1}. 
\end{lemma}

\subsection{Existence and uniqueness of solutions}

Let us briefly verify that \eqref{OUsemigroup} is a solution to \eqref{FPEq1}, or equivalently: 
\begin{lemma}
For $\phi_0\in \Cvx(\R^n)$,
\begin{equation}
\label{FPsolution}
    \phi(t,x)\defeq - \log\int_{\R^n} e^{-\frac{|x- e^{-t}y|^2}{2(1-e^{-2t})}}\frac{e^{-\phi_0(y)} dy}{(2\pi(1-e^{-2t}))^{n/2}}
\end{equation}
is the unique solution to \eqref{FPEq2}. 
\end{lemma}
\begin{proof}
    Let
    \begin{equation}
    \label{gammanDef}
        \gamma_{n}(\beta, x)\defeq \frac{1}{(2\pi\beta)^{n/2}} e^{-\frac{|x|^2}{2\beta}}, \quad \beta>0, x\in \R^n. 
    \end{equation}
    This is such that $\int \gamma_n(x,\beta)dx=1$ for all $\beta>0$. Using $\gamma_{n}$, \eqref{FPsolution} may be expressed as 
    \begin{equation*}
        \phi(t,x)= -\log\int_{\R^n} e^{-\phi_0(y)} \gamma_n(1-e^{-2t}, x-e^{-t}y) dy. 
    \end{equation*}
    By Lemma \ref{FPequivalence}, and the linearity of the integral, in order to prove that $\phi$ solves \eqref{FPEq2}, it is enough to demonstrate that for all $y\in\R^n$, 
    \begin{equation*}
    G_y(t,x)\defeq  \gamma_n(1-e^{-2t}, x-e^{-t}y) 
    \end{equation*} 
    satisfies \eqref{FPEq1}. 

    To see why $G_y$ satisfies \eqref{FPEq2}, 
    note that 
    \begin{equation}
    \label{gammanDerivatives}
        \begin{aligned}
            & \partial_\beta \gamma_n(\beta, x)= \frac{|x|^2- n\beta}{2\beta^2}\gamma_n(\beta, x),  \\
            & D\gamma_n(\beta, x) = -\frac{x}{\beta}\gamma_n(\beta, x), \\
            & \Delta\gamma_n(\beta, x)= \frac{|x|^2- n\beta}{\beta^2}\gamma_n(\beta, x). 
        \end{aligned}
    \end{equation}
    Therefore, using \eqref{gammanDerivatives} and that $\beta = 1-e^{-2t}$, 
    {\allowdisplaybreaks
    \begin{equation}
    \label{partialtG}
    \begin{aligned}
            \partial_t G_y &= \partial_t(1-e^{-2t}) \partial_\beta \gamma_n + \langle D\gamma_n, \partial_t(x-e^{-t}y)\rangle \\ 
            &= 2e^{-2t} \frac{|x-e^{-2t}y|^2 - n(1-e^{-2t})}{2(1-e^{-2t})^2}\gamma_n + \langle -\frac{x-e^{-t}y}{1-e^{-2t}}\gamma_n, e^{-t}y\rangle\\
            &= e^{-2t} \frac{|x-e^{-2t}y|^2 - n(1-e^{-2t})}{(1-e^{-2t})^2}\gamma_n + \langle \frac{x-e^{-t}y}{1-e^{-2t}}, x-e^{-t}y-x\rangle \gamma_n\\
            &= \left[\frac{e^{-2t} |x-e^{-2t}y|^2}{(1-e^{-2t})^2} - \frac{ne^{-2t}}{1-e^{-2t}} + \frac{|x-e^{-t}y|^2- \langle x-e^{-t}y, x\rangle}{1-e^{-2t}}\right]\gamma_n\\
            &= \left[\frac{e^{-2t} |x-e^{-2t}y|^2}{(1-e^{-2t})^2} - \frac{ne^{-2t}}{1-e^{-2t}} + \frac{(1-e^{-2t})|x-e^{-t}y|^2}{(1-e^{-2t})^2} -\frac{\langle x-e^{-t}y, x\rangle}{1-e^{-2t}}\right]\gamma_n \\
            &= \left[\frac{|x-e^{-2t}y|^2}{(1-e^{-2t})^2} - \frac{ne^{-2t}}{1-e^{-2t}} -\frac{\langle x-e^{-t}y, x\rangle}{1-e^{-2t}}\right]\gamma_n.
    \end{aligned}
    \end{equation}
    }
    In addition, by \eqref{gammanDerivatives}, 
    \begin{equation}
    \label{GyDerivative}
        \begin{aligned}
            & DG_y = D\gamma_n (1-e^{-2t}, x- e^{-t}y)= -\frac{x-e^{-t}y}{1-e^{-2t}} \gamma_n, \\
            & \Delta G_y= \left( -\frac{n}{1-e^{-2t}} + \frac{|x-e^{-t}y|^2}{(1-e^{-2t})^2}\right) \gamma_n. 
        \end{aligned}
    \end{equation}
    Therefore, by \eqref{GyDerivative} and \eqref{partialtG}, 
    \begin{equation*}
        \begin{aligned}
            &\Delta G_y + \langle x, DG_y\rangle + n G_y \\
            &= \left[ -\frac{n}{1-e^{-2t}} + \frac{|x-e^{-t}y|^2}{(1-e^{-2t})^2} - \frac{\langle x, x-e^{-t}y\rangle}{1-e^{-2t}}+ n\right] \gamma_n\\
            &= \left[ -\frac{e^{-2t} n}{1-e^{-2t}} + \frac{|x-e^{-t}y|^2}{(1-e^{-2t})^2} - \frac{\langle x, x-e^{-t}y\rangle}{1-e^{-2t}}\right]\gamma_n \\
            &= \partial_tG_y. 
        \end{aligned}
    \end{equation*}
    
    Since \eqref{FPEq2} is a second-order strictly parabolic PDE, uniqueness follows from standard PDE theory \cite[Corollary 5.2]{Lieberman96}. 
\end{proof}

\begin{lemma}
\label{LimitingSolution}
    Let $\phi$ be a solution to \eqref{FPEq2}. Then, 
    \begin{equation*}
        \phi(\infty, x)\defeq \lim_{t\to\infty}\phi(t,x)= |x|^2/2 - \log\frac{V(\phi)}{(2\pi)^{n/2}}. 
    \end{equation*}
\end{lemma}
\begin{proof}
By taking the limit in \eqref{FPsolution}, 
\begin{equation*}
    \lim_{t\to\infty}\phi(t,x) = -\log\int_{\R^n} e^{-\frac{|x|^2}{2}}e^{-\phi_0(y)}\frac{dy}{(2\pi)^{n/2}} = |x|^2/2- \log\frac{V(\phi)}{(2\pi)^{n/2}}. 
\end{equation*}
proves the claim. 
\end{proof}

\subsection{Evolution equations}

\subsubsection{Volume}
The volume of a function that evolves under \eqref{FPEq2} remains constant. One way of doing this is by using Stoke's theorem. However, since we have an explicit integral for the solution \eqref{FPsolution}, we may directly verify that $V(\phi(t,\,\cdot\,))= V(\phi_0)$ for all $t$. 
\begin{lemma}
\label{VolEvolution}
    The volume of a convex function evolving under \eqref{FPEq2} is constant in time. 
\end{lemma}
\begin{proof}
    By \eqref{FPsolution}, and for $\gamma_n$ as in \eqref{gammanDef},
    \begin{equation*}
    \begin{aligned}
        V(\phi(t,\,\cdot\,))&= \int_{\R^n}\int_{\R^n} e^{-\phi(y)} \gamma_n(1-e^{-2t}, x- e^{-t}y)dy dx \\
        &= \int_{\R^n} e^{-\phi_0(y)}\int_{\R^n}\gamma_n(1-e^{-2t}, x-e^{-t}y)dx dy \\
        &= \int_{\R^n} e^{-\phi_0(y)}dy = V(\phi_0),
    \end{aligned}
    \end{equation*}
    because $\int \gamma_n(\beta, x)dx = 1$ for all $\beta>0$. 
\end{proof}

\subsubsection{\texorpdfstring{$L^p$-Legendre transform}{Lp-Legendre transform}}
Let us now compute the evolution equation of the $L^p$-Legendre transform of a convex function $\phi$ satisfying \eqref{FPEq2}. By this, we mean the evolution equation of $\phi(t,\,\cdot\,)^{*,p}$, where the transform is only with respect to the spacial direction. For clarity, denote by 
\begin{equation*}
    \phi_t(x)\defeq \phi(t,x), \quad x\in \R^n, 
\end{equation*}
so that $\phi^{*,p}(t,\,\cdot\,)= \phi_t^{*,p}$. 

On another note, as we see below, the evolution equation of $\phi_t^{*,p}$ involves the Fischer information of the following probability measure: 
\begin{equation}
\label{partialtphi}
    d\phi^{p,y}(x) \defeq \frac{e^{p\langle x,y\rangle - (p+1)\phi(x)}dx}{\int_{\R^n}e^{p\langle x,y\rangle- (p+1)\phi(x)}dx}. 
\end{equation}
Recall the Fischer information of a probability measure $\mu(x)= f(x)dx$ is given by 
\begin{equation*}
    I(\mu)\defeq \int_{\R^n}|D\log f(x)|^2 f(x)dx = \int_{\R^n}\frac{|Df(x)|^2}{f(x)}dx. 
\end{equation*}
For now, the definition suffices. In the next section, we recall a few more properties of the Fischer information.

\begin{lemma}
\label{LpLegendreEvolution}
    Let $p\in (0,\infty)$. For $\phi(t, x)$ evolving under \eqref{FPEq2},
    \begin{equation*}
        \partial_t \phi^{*,p}(t,y) = \frac{p}{p+1}|y|^2 - \langle y, D\phi^{*,p}(t,y)\rangle - \frac{I(d\phi^{p,y})}{p+1} + n. 
    \end{equation*}
\end{lemma}
To simplify the proof of Lemma \ref{LpLegendreEvolution}, let
\begin{equation}
\label{fFDef}
    f(t,x) = e^{-\phi(t,x)} \quad \text{ and } \quad F(t,x)\defeq f(t,x)^{p+1},
\end{equation} 
so that
$
    \phi^{*,p}(t,y)= \frac1p \log\int_{\R^n}F(t,x)e^{p\langle x,y\rangle}\frac{dx}{V(\phi)}. 
$
\begin{claim}
\label{FEEq}
Let $p\in (0,\infty)$ and $\phi$ evolving under \eqref{FPEq2}. For $F$ as in \eqref{fFDef}, 
\begin{equation*}
    \partial_t F = \Delta F - \frac{p}{p+1}\frac{|DF|^2}{F} + \langle x, DF\rangle + n(p+1)F. 
\end{equation*}
\end{claim}
\begin{proof}
    For $f$ and $F$ as in \eqref{fFDef}, 
    \begin{align}
            &DF = (p+1)f^{p}Df \\
            &\Delta F = p(p+1)f^{p-1}|Df|^2 + (p+1)f^p\Delta f. 
    \end{align}
    because $\Delta F = \sum_{i=1}^n\partial_{x_i}(\partial_{x_i}F)= (p+1)\sum_{i=1}^n\partial_{x_i}(f^p \partial_{x_i}f) = (p+1)\sum_{i=1}^n pf^{p-1}(\partial_{x_i}f)^2 + f^p\partial_{x_i}^2 f$. 
    Solving for $f$, 
    \begin{align}
        Df &= \frac{DF}{(p+1)f^p} \label{DFeq} \\
        \Delta f &= \frac{1}{(p+1)f^{p}}\left(\Delta F - \frac{p}{p+1}\frac{|DF|^2}{F}\right) \label{DeltaFEq}.
    \end{align}
    By Lemma \ref{FPequivalence}, since $\phi$ evolves under \eqref{FPEq2}, $f$ evolves under \eqref{FPEq1}. Therefore,  $ \partial_t F = (p+1)f^p \partial_t f= (p+1)f^p(\Delta f + \langle x, Df\rangle + nf)$. By \eqref{DeltaFEq}, $(p+1)f^p\Delta f = \Delta F - \frac{p}{p+1}\frac{|DF|^2}{F}$, and by \eqref{DFeq}, $(p+1)f^p\langle x, Df\rangle = \langle x, DF\rangle$. Finally, $(p+1)f^p (nf)= n(p+1)f^{p+1}= n(p+1)F$, hence the claim. 
\end{proof}

\begin{claim}
\label{IdphiClaim}
    Let $p\in (0,\infty)$ and $\phi\in\Cvx(\R^n)$. For $F$ as in \eqref{fFDef} and $d\phi^{p,y}$ as in \eqref{partialtphi}, 
    \begin{equation*}
        I(d\phi^{p,y})= e^{-p\phi^{*,p}(y)} \int_{\R^n}\frac{|DF|^2}{F} e^{p\langle x,y\rangle}\frac{dx}{V(\phi)} -p^2 |y|^2. 
    \end{equation*}
\end{claim}
\begin{proof}
    First, note that in terms of $F$, $d\phi^{p,y}(x)= \frac{F(x)e^{p\langle x, y\rangle}}{V(\phi)e^{p\phi^{*,p}(y)}}$, hence by definition of the Fischer information, 
    \begin{equation}
    \label{IdphiEq1}
        I(d\phi^{p,y})= e^{-p\phi^{*,p}(y)}\int_{\R^n} |D\log(F(x)e^{p\langle x,y\rangle})|^2 F(x)e^{p\langle x,y\rangle}\frac{dx}{V(\phi)}. 
    \end{equation}
    Since 
    \begin{equation}
    \label{IdphiEq2}
        |D\log(F(x)e^{p\langle x,y\rangle})|^2 = |D\log F + py|^2= |D\log F|^2 + p^2 |y|^2 + 2p\langle D\log F, y\rangle, 
    \end{equation}
    we can decompose 
    $I(d\phi^{p,y})$ into three integrals:
    \begin{align*}
        &I_1 \defeq e^{-p\phi^{*,p}(y)} \int_{\R^n}|D\log F(x)|^2 F(x) e^{p\langle x,y\rangle}\frac{dx}{V(\phi)}, \\
        &I_2 \defeq p^2 |y|^2 e^{-p\phi^{*,p}(y)} \int_{\R^n}F(x) e^{p\langle x,y\rangle}\frac{dx}{V(\phi)}, \\
        &I_3 \defeq 2p e^{-p\phi^{*,p}(y)} \int_{\R^n}\langle D\log F, y\rangle F(x)e^{p\langle x,y\rangle}\frac{dx}{V(\phi)},
    \end{align*}
    where, by \eqref{IdphiEq1} and \eqref{IdphiEq2}, 
    \begin{equation}
    \label{IdphiDecomp}
        I(d\phi^{p,y})= I_1 + I_2 + I_3. 
    \end{equation}

    \noindent
    For $I_1$: Leave it as is. 

    \noindent
    For $I_2$: 
    \begin{equation}
    \label{IdphiEq3}
        I_2 = p^2 |y|^2 e^{-p\phi^{*,p}(y)} \int_{\R^n}e^{p\langle x,y\rangle- (p+1)\phi(x)}\frac{dx}{V(\phi)} = p^2 |y|^2 e^{-p\phi^{*,p}(y)} e^{p\phi^{*,p}(y)} = p^2 |y|^2. 
    \end{equation}
    For $I_3$: By integration by parts, 
    \begin{equation}
    \label{IdphiEq4}
        \begin{aligned}
            I_3 &= 2p e^{-p\phi^{*,p}(y)} \int_{\R^n}\langle y, D\log F\rangle Fe^{p\langle x,y\rangle}\frac{dx}{V(\phi)} \\
            &= 2p e^{-p\phi^{*,p}(y)} \int_{\R^n}\langle y, \frac{DF}{F}\rangle F e^{p\langle x,y\rangle}\frac{dx}{V(\phi)}\\
            &= 2p e^{-p\phi^{*,p}(y)}\int_{\R^n} \langle y, DF\rangle e^{p\langle x, y\rangle}\frac{dx}{V(\phi)} \\
            &= 2p e^{-p\phi^{*,p}(y)}\int_{\R^n} (-p|y|^2) F(x) e^{p\langle x,y\rangle}\frac{dx}{V(\phi)}\\
            &= -2p^2|y|^2.  
        \end{aligned}
    \end{equation}
    The claim follows from \eqref{IdphiDecomp}--\eqref{IdphiEq4}. 
\end{proof}

\begin{proof}[Proof of Lemma \ref{LpLegendreEvolution}]
    By definition of $F$ \eqref{fFDef}, $\phi^{*,p}(t,y)= \frac{1}{p}\log\int_{\R^n}F(t,x)e^{p\langle x,y\rangle}\frac{dx}{V(\phi)}$. Therefore, differentiating in $t$, recalling that $V(\phi)$ remains constant in time (Lemma \ref{VolEvolution}), 
    \begin{equation*}
        \partial_t \phi^{*,p}(t,y)= \frac{\int_{\R^n}\partial_t F(t,x)e^{p\langle x, y\rangle}\frac{dx}{V(\phi)}}{p\int_{\R^n}F(t,x)e^{p\langle x,y\rangle}\frac{dx}{V(\phi)}}
        = \frac{\int_{\R^n}\partial_t F(t,x)e^{p\langle x, y\rangle} dx}{p\int_{\R^n}F(t,x)e^{p\langle x,y\rangle} dx}. 
    \end{equation*}
    Therefore, by Claim \ref{FEEq}, 
    \begin{equation}
    \label{dtPhi}
        \partial_t \phi^{*,p}(t,y) = \frac{\int_{\R^n}(\Delta F - \frac{p}{p+1}\frac{|DF|^2}{F}+ \langle x, DF\rangle + n(p+1)F) e^{p\langle x,y\rangle}\frac{dx}{V(\phi)}}{p\int_{\R^n} F(t,x)e^{p\langle x,y\rangle}\frac{dx}{V(\phi)}}\defeq \frac{I_1+ I_2 + I_3+  I_4}{pe^{p\phi^{*,p}(t,y)}}. 
    \end{equation}
    For $I_1$: Integrating by parts, 
    \begin{equation}
    \label{I1}
        \begin{aligned}
            I_1&\defeq \int_{\R^n}\Delta F\, e^{p\langle x,y\rangle}\frac{dx}{V(\phi)}= \sum_{i=1}^n \int_{\R^n}\frac{\partial^2 F}{\partial x_i^2}e^{p\langle x,y\rangle}\frac{dx}{V(\phi)} = \sum_{i=1}^n -\int_{\R^n}\frac{\partial F}{\partial x_i}p y_i e^{p\langle x,y\rangle}\frac{dx}{V(\phi)} \\
            &= \sum_{i=1}^n \int_{\R^n}F p^2 y_i^2 e^{p\langle x,y\rangle}\frac{dx}{V(\phi)}= p^2 \sum_{i=1}^n y_i^2 \int_{\R^n}F e^{p\langle x,y\rangle}\frac{dx}{V(\phi)}\\
            &= p^2 |y|^2 e^{p\phi^{*,p}(t,y)}. 
        \end{aligned}
    \end{equation}
    For $I_2$: By Claim \ref{IdphiClaim}, 
    \begin{equation}
    \label{I2eq1}
        \begin{aligned}
            I_2\defeq -\frac{p}{p+1}\int_{\R^n}\frac{|DF|^2}{F}e^{p\langle x,y\rangle}\frac{dx}{V(\phi)} = -\frac{p}{p+1}(I(d\phi^{p,y}) + p^2 |y|^2)e^{p,\phi^{*,p}(y)}. 
        \end{aligned}
    \end{equation}
    For $I_3$: Integrating by parts, 
    \begin{equation}
    \label{I3}
        \begin{aligned}
            I_3&\defeq \int_{\R^n}\langle x, DF\rangle e^{p\langle x,y\rangle}\frac{dx}{V(\phi)}\\
            &= -\int_{\R^n} n F e^{p\langle x,y\rangle} + p\langle x,y\rangle F e^{p\langle x, y\rangle}\frac{dx}{V(\phi)} \\
            &= -n e^{p\phi^{*,p}(t,y)} - p\langle y, D\phi^{*,p}(t,y)\rangle e^{p\phi^{*,p}(t,y)}. 
        \end{aligned}
    \end{equation}
    For $I_4$:  
    \begin{equation}
    \label{I4}
        I_4\defeq \int_{\R^n}n(p+1)F e^{p\langle x,y\rangle}\frac{dx}{V(\phi)} = n(p+1)e^{p\phi^{*,p}(t,y)}. 
    \end{equation}
    Consequently, by \eqref{dtPhi}--\eqref{I4}, 
    \begin{equation*}
    \begin{aligned}
        \partial_t \phi^{*,p}(t,y)&= p|y|^2 -\frac{I(Fe^{p\langle \cdot, y\rangle})}{p+1}- \frac{p^2}{p+1}|y|^2- \frac{n}{p} - \langle y, D\phi^{*,p}(t,y)\rangle + \frac{n(p+1)}{p}\\
        &= \frac{p}{p+1}|y|^2 - \langle y, D\phi^{*,p}(t,y)\rangle - \frac{I(Fe^{p\langle \cdot, y\rangle})}{p+1} + n, 
    \end{aligned}
    \end{equation*}
    as desired. 
\end{proof}

\subsubsection{\texorpdfstring{$L^p$-Mahler integral}{Lp-Mahler integral}}
\begin{lemma}
\label{LpMahlerIntegralEvolution}
    Let $p\in (0,\infty)$. For $\phi(t,x)$ evolving under \eqref{FPEq2}, 
    \begin{equation*}
        \partial_t \M_p(\phi)= \frac{p \M_p(\phi)}{p+1}\left(\frac1p\int_{\R^n}I(d\phi^{p,y})\frac{e^{-\phi^{*,p}(t,y)}dy}{V(\phi^{*,p})} -\mathrm{tr}\,\mathrm{Cov}(\phi^{*,p}) - |b(\phi^{*,p})|^2\right). 
    \end{equation*}
\end{lemma}
One of the terms that will appear in the evolution equation of the $L^p$-Mahler integral is the covariance matrix of the measure $d\phi^{p,y}$. 
Recall the covariance matrix of a probability measure $\mu$, 
\begin{equation*}
    \mathrm{Cov}(\mu)\defeq \left[\int x_ix_jd\mu(x) - \int x_i d\mu(x)\int x_jd\mu(x)\right]_{i,j=1}^n. 
\end{equation*}
For a function $f$, let $\mathrm{Cov}(f)\defeq \mathrm{Cov}(\frac{f(x)dx}{V(f)})$. 
\begin{proof}[Proof of Lemma \ref{LpMahlerIntegralEvolution}]
    Since, by Lemma \ref{VolEvolution}, $\partial_tV(\phi)=0$, it suffices to compute $\partial_t V(\phi^{*,p})$. By Lemmas \ref{LpLegendreEvolution} and \ref{D^2CovEq}, 
    \begin{equation*}
        \begin{aligned}
            \partial_t V(\phi^{*,p})&= -\int_{\R^n}\partial_t\phi^{*,p}(t,y) e^{-\phi^{*,p}(t,y)}dy \\
            &= \int_{\R^n}\left( -\frac{p}{p+1}|y|^2 + \langle y, D\phi^{*,p}\rangle + \frac{I(d\phi^{p,y})}{p+1} - n\right) e^{-\phi^{*,p}(y)}dy \\
            &= I_1 + I_2 + I_3 + I_4. 
        \end{aligned}
    \end{equation*}
    For $I_1$: 
    \begin{equation*}
        \begin{aligned}
            I_1&\defeq -\frac{p}{p+1}\int_{\R^n}|y|^2 e^{-\phi^{*,p}(t,y)}dy = -\frac{p V(\phi^{*,p})}{p+1} \int_{\R^n} |y|^2 \frac{e^{-\phi^{*,p}(t,y)}dy}{V(\phi^{*,p})} \\
            &= -\frac{pV(\phi^{*,p})}{p+1}\left( \mathrm{tr}\,\mathrm{Cov}(\phi^{*,p}) + \sum_{i=1}^n\left( \int_{\R^n}y_i\frac{e^{-\phi^{*,p}(t,y)}dy}{V(\phi^{*,p})}\right)^2\right) \\
            &= -\frac{pV(\phi^{*,p})}{p+1}\left( \mathrm{tr}\,\mathrm{Cov}(\phi^{*,p}) + |b(\phi^{*,p})|^2\right). 
        \end{aligned}
    \end{equation*}
    For $I_2$: Integrating by parts, 
    \begin{equation*}
        I_2\defeq \int_{\R^n}\langle y, D\phi^{*,p}\rangle e^{-\phi^{*,p}(y)}dy = -\int_{\R^n}\langle y, D(e^{-\phi^{*,p}})\rangle dy = n \int_{\R^n} e^{-\phi^{*,p}(t,y)}dy= n V(\phi^{*,p}). 
    \end{equation*}
    For $I_3$: 
    \begin{equation*}
        I_3 \defeq \frac{1}{p+1}\int_{\R^n} I(d\phi^{p,y}) e^{-\phi^{*,p}(t,y)}dy =\frac{V(\phi^{*,p})}{p+1}\int_{\R^n} I(d\phi^{p,y})\frac{e^{-\phi^{*,p}(t,y)}dy}{V(\phi^{*,p})}. 
    \end{equation*}
    For $I_4$: 
    \begin{equation*}
        I_4\defeq \int_{\R^n}(-n)e^{-\phi^{*,p}(t,y)}dy = -n V(\phi^{*,p}). 
    \end{equation*}
    Summing up the four integrals proves the claim. 
\end{proof}

\subsection{\texorpdfstring{Proving the functional $L^p$-Santal\'o inequality}{Proving the functional Lp-Santal\'o inequality}}
\subsubsection{An inequality for Fischer information}
The Fischer information matrix of a probability measure
\begin{equation*}
    \mathcal{I}(\mu)\defeq \int_{\R^n}\frac{Df(Df)^T}{f}d\mu(x)
\end{equation*}
is such that $\mathrm{tr}\,\mathcal{I}(\mu) = I(\mu)$. 
\begin{theorem}[Cram\'er--Rao]
\label{CramerRaoIneq}
For a probability measure $\mu$, $\mathcal{I}(\mu)\geq \mathrm{Cov}(\mu)^{-1}$. 
\end{theorem}

The following generalizes our past observation \cite[Lemma 6.3]{BMR23}:
\begin{lemma}
\label{D^2CovEq}
    For $p\in (0,\infty)$, $D^2 \phi^{*,p}(y)= p\,\mathrm{Cov}(d\phi^{p,y})$. 
\end{lemma}
\begin{proof}
    Let us start with the first derivative: 
    \begin{equation*}
        \frac{\partial \phi^{*,p}}{\partial y_i}(y) = \frac{\int_{\R^n}px_i e^{p\langle x,y\rangle- (p+1)\phi(x)}\frac{dx}{V(\phi)}}{p\int_{\R^n} e^{p\langle x,y\rangle -(p+1)\phi(x)}\frac{dx}{V(\phi)}} = \frac{\int_{\R^n} x_i e^{p\langle x,y\rangle- (p+1)\phi(x)}dx}{\int_{\R^n}e^{p\langle x,y\rangle- (p+1)\phi(x)}dx}. 
    \end{equation*}
    Continuing with the second derivative: 
    \begin{equation*}
        \begin{aligned}
            \frac{\partial^2\phi^{*,p}}{\partial y_i\partial y_j}(y) &= \frac{p\int_{\R^n}x_ix_j e^{p\langle x,y\rangle-(p+1)\phi(x)}dx}{{\int_{\R^n}e^{p\langle x,y\rangle - (p+1)\phi(x)}dx}} - \frac{p \int_{\R^n}x_i e^{p\langle x,y\rangle- (p+1)\phi(x)}dx \int_{\R^n}x_je^{p\langle x,y\rangle - (p+1)\phi(x)}dx}{\left( \int_{\R^n}e^{p\langle x,y\rangle - (p+1)\phi(x)}dx\right)^2}\\
            &= p\int_{\R^n}x_ix_j d\phi^{p,y}(x) - p \int_{\R^n}x_i d\phi^{p,y}(x)\int_{\R^n}x_j e^{p\langle x,y\rangle} dx \\
            &= p\mathrm{Cov}(d\phi^{p,y}). 
        \end{aligned}
    \end{equation*}
\end{proof}

\begin{corollary}
\label{CramerRaoCorollary}
    Let $p\in (0,\infty)$. For $\phi\in \Cvx(\R^n)$, $I(d\phi^{p,y})\geq p\,\mathrm{tr}(D^2\phi^{*,p}(y)^{-1})$.  
\end{corollary}
\begin{proof}
    By the Cram\`er--Rao inequality and Lemma \ref{D^2CovEq}, $I(d\phi^{p,y})= \mathrm{tr}\,\mathcal{I}(d\phi^{p,y})\geq \mathrm{tr}\,\mathrm{Cov}(d\phi^{p,y})^{-1}= \mathrm{tr}\,(p^{-1}D^2\phi^{*,p}(y))^{-1}= p \,\mathrm{tr}\,(D^2\phi^{*,p}(y)^{-1})$. 
\end{proof}

\subsubsection{An inequality for the trace of Covariance}
To get a bound on the integral of the Fischer information we need the following variance Brascamp--Lieb inequality \cite{BrascampLieb76b}: 
\begin{lemma}
\label{VarianceIneq}
    For a strictly convex probability function $\psi$ and a locally Lipschitz $g\in L^1(e^{-\psi})$, 
    \begin{equation*}
        \int_{\R^n}|g(y)|^2\frac{e^{-\psi(y)}dy}{V(\psi)} -\left(\int_{\R^n}g(y)\frac{e^{-\psi(y)}dx}{V(\psi)} \right)^2\leq \int_{\R^n}\langle Dg(y), D^2\psi(y)^{-1}Dg(y)\rangle \frac{e^{-\psi(y)dy}}{V(\psi)}. 
    \end{equation*}
\end{lemma}
\begin{corollary}
\label{traceInformationIneq}
    For $p\in (0,\infty)$ and $\phi\in\Cvx(\R^n)\cap C^2$, 
    \begin{equation*}
        \mathrm{tr}\,\mathrm{Cov}(\phi^{*,p})\leq \frac1p \int_{\R^n} I(d\phi^{p,y})\frac{e^{-\phi^{*,p}(y)}dy}{V(\phi^{*,p})}
    \end{equation*}
\end{corollary}
\begin{proof}
Applyng Lemma \ref{VarianceIneq} for $\psi = \phi^{*,p}$ and $g= y_i$ yields: 
\begin{equation*}
    \int_{\R^n}y_i^2 \frac{e^{-\phi^{*,p}(y)}dy}{V(\phi^{*,p})} - \left( \int_{\R^n}y_i \frac{e^{-\phi^{*,p}(y)}dy}{V(\phi^{*,p})}\right)^2 \leq \int_{\R^n}\langle e_i, D^2\phi^{*,p}(y)^{-1}e_i\rangle \frac{e^{-\phi^{*,p}(y)}dy}{V(\phi^{*,p})}
\end{equation*}
Summing over all $i$, 
\begin{equation*}
\label{trIneq}
    \mathrm{tr}\,\mathrm{Cov}(\phi^{*,p})\leq \int_{\R^n} \mathrm{tr}(D^2\phi^{*,p}(y)^{-1})\frac{e^{-\phi^{*,p}(y)}dy}{V(\phi^{*,p})},
\end{equation*}
hence the claim follows from Corollary \ref{CramerRaoCorollary}. 
\end{proof}

\subsubsection{Finishing the proof}
\label{FinishingProofSection}
\begin{proof}[Proof of Proposition \ref{Mpbound}]
    Follows directly from the evolution equation $\partial_t \M_p(\phi)$ and the inequality of Corollary \ref{traceInformationIneq}. 
\end{proof}

\begin{proof}[Proof of Theorem \ref{LpFuncSantalo}]
Fix $\phi_0\in\Cvx(\R^n)$, and let $\phi(t,x), t>0,$ be the unique solution to \eqref{FPEq2} with $\phi(0,x)= \phi_0(x)$. Let  
\begin{equation*}
    g(t)\defeq \M_p(T_{s_p(\phi)}\phi)= V(T_{s_p(\phi)}\phi) V((T_{s_p(\phi)}\phi)^{*,p})= V(\phi) \int_{\R^n}e^{-\phi^{*,p}(t,y)} e^{\langle s_p(\phi), y\rangle}dy 
\end{equation*}
By the chain rule: 
\begin{equation*}
    g'(t)= \partial_t \M_p(T_{s_p(\phi)}\phi) + \langle \partial_t s_p(\phi), (D_x\M_p(T_x\phi))(s_p(\phi))\rangle. 
\end{equation*}
By definition, the $L^p$-Santal\'o point minimizers $\M_p(T_x\phi)$ among all $x\in \R^n$. Consequently, $(D_x\M_p(T_x\phi))(s_p(\phi))=0$. In addition, the $L^p$-Santal\'o point is characterized by $b((T_{s_{p}(\phi)}\phi)^{*,p})=0$ (Proposition \ref{LpFunSantaloPoint}). Therefore, by Proposition \ref{Mpbound}, $\partial_t \M_p(T_{s_p(\phi)}\phi)\geq - |b((T_{s_{p}(\phi)}\phi)^{*,p})|^2=0$. As a result, $g'(t)\geq0$, i.e., $g$ is an increasing function.  Since, by Lemma \ref{LimitingSolution}, $\phi_\infty(x)\defeq \lim_{t\to\infty}$ is $|x|^2/2$ minus a constant, it follows that 
\begin{equation*}
    \inf_{x\in \R^n}\M_p(\phi_0)= \M_p(T_{s_{p}(\phi_0)}\phi_0)= g(0)\leq g(\infty)= \M_p(|x|^2/2), 
\end{equation*}
as claimed. 
\end{proof}

\begin{proof}[Proof of Corollary \ref{GeneralLpFunSantalo}]
    The same proof as in Theorem \ref{LpFuncSantalo} works without the convexity assumption: The Fokker--Planck flow is instantaneously smoothing, Lemma \ref{CramerRaoCorollary} holds without the convexity assumption on $\phi$, and Lemma \ref{VarianceIneq} is applied to $f^{*,p}$ which is always convex. Super-linearity is essential to ensure the validity of the integration by parts used throughout. 
\end{proof}

\appendix
\section{Moments of a convex function}
Throughout this paper, we have focused on convex functions with non-zero and finite volume. Although this condition naturally arises as a generalization of the assumptions typically made for convex bodies, it comes with the additional benefit of ensuring that all integrals considered in this paper are convergent. Specifically, the $k$-th moments of the function are well-defined and finite. This results directly from the following lemma \cite[Lemma 2.2.1]{BGV+14}.  
 
\begin{lemma}
\label{ConvexLowerBoundLemma}
    Let $\phi:\R^n\to\R\cup\{\infty\}$ be a convex function with $0<V(\phi)<\infty$. There exist constants $a>0, b\in \R$ such that $\phi(x)\geq a|x| + b$ for all $x\in \R^n$. 
\end{lemma}
\begin{proof}
\textit{Step 1}: Since $V(\phi)>0$, the set $\{\phi<\infty\}$ has positive Lebesgue measure. By writing, 
\begin{equation*}
    \{\phi<\infty\}= \bigcup_{m\geq 0}\{\phi\leq m\}
\end{equation*}
there must exist $m_0\in \mathbb{N}$ such that $|\{\phi\leq m_0\}|>0$. 

\smallskip
\noindent
\textit{Step 2}: $\{\phi\leq m_0\}$ has non-empty interior. Indeed, since $|\{\phi\leq m_0\}|>0$, the set $\{\phi\leq m_0\}$ cannot be contained within a hyperplane and must contain linearly independent points $x_1, \ldots, x_n$. Fix a point $x_0\in \{\phi\leq m_0\}\setminus\{x_1, \ldots, x_n\}$. The convex hull of these points,
\begin{equation*}
    \Delta\defeq \mathrm{conv}\{x_0, x_1, \ldots, x_n\}
\end{equation*}
is a linear image of $\Delta_{n,+}$. One can see this by translating $\Delta- x_0= \mathrm{conv}\{0, x_1-x_0, \ldots, x_n-x_0\}$. The points $x_1- x_0, \ldots, x_n- x_0$ are still linearly independent, hence there exists $A\in GL(n,\R)$ such that $A(x_i-x_0)= e_i$. Therefore, $A(\Delta-x_0)= \Delta_{n,+}$ or equivalently  $\Delta= A^{-1}\Delta_{n,+} + x_0$, meaning it has non-empty interior. Since $\{\phi\leq m_0\}$ is convex and contains $x_0,\ldots, x_n$, it must also contain $\Delta$, and hence has a non-empty interior.  

\smallskip
\noindent 
\textit{Step 3:} The set $\{\phi\leq m_0+1\}$ is bounded. To prove this, recall that for convex sets with non-empty interior, boundedness is equivalent to finiteness of volume (Claim \ref{FiniteVolClaim} below). Therefore, it is enough to show that $\{\phi\leq m_0+1\}$ has a non-empty interior and finite volume. 

The non-emptiness of the interior follows directly from Step 2, as $\{\phi\leq m_0\}$, which has non-empty interior, is contained in $\{\phi\leq m_0+1\}$. For the boundedness of the volume, use Markov's inequality,
\begin{equation*}
    |\{\phi\leq m_0+ 1\}|= \int_{\{\phi\leq m_0+1\}} dx \leq e^{m_0+1}\int_{\{\phi\leq m_0+1\}}e^{-\phi(x)} dx \leq e^{m_0+1} V(\phi).
\end{equation*}
Since $V(\phi)$ is finite by assumption, $|\{\phi\leq m_0 + 1\}|$ is also finite. 
Consequently, $\{\phi\leq m_0+1\}$ is a convex set with non-empty interior and finite volume. By Claim \ref{FiniteVolClaim} below, it is bounded. 

\smallskip
\noindent
\textit{Step 4}: Upper bound for small $|x|$. By Step 2, there exists $r>0$ such that a ball of radius $r$ is contained in $\{\phi\leq m_0\}$. By replacing $\phi$(x) with $\phi(x- x_0)$, we may assume $rB_2^n\subset \{\phi\leq m_0\}$. Therefore, 
\begin{equation}
\label{AppendixEq1}
    \phi(x)\leq m_0, \quad \text{ for } |x|\leq r,
\end{equation}
providing an upper bound for small $|x|$. 

\smallskip 
\noindent
\textit{Step 5}: Upper bound for big $|x|$. 
By Step 3, there exists $R>0$ such that $\{\phi\leq m_0+1\}\subset \frac{R}{2}B_2^n$. That is, 
\begin{equation}
\label{AppendixEq2}
    \phi(x) > m_0 + 1, \quad \text{ for } |x|\geq R. 
\end{equation}
For $|x|>R$, we may express $Rx/|x|$ as a convex combination (Figure \ref{ConvexCombFig}): 
\begin{equation*}
    \frac{Rx}{|x|} = \frac{|x|-R}{|x|-r} \frac{rx}{|x|} + \frac{R-r}{|x|-r} x
\end{equation*}
By convexity of $\phi$, and bounds \eqref{AppendixEq1} and \eqref{AppendixEq2},
\begin{equation*}
    m_0+ 1 < \phi\left(\frac{Rx}{|x|}\right)\leq \frac{|x|-R}{|x|-r} \phi\left( \frac{rx}{|x|}\right) + \frac{R-r}{|x|-r}\phi(x)\leq \frac{|x|-R}{|x|-r} m_0 + \frac{R-r}{|x|-r}\phi(x). 
\end{equation*}
Therefore, 
\begin{equation}
\label{AppendixEq3}
    \phi(x)\leq \frac{|x|}{R-r} + m_0 - \frac{r}{R-r}, \quad \text{ for } |x|>R. 
\end{equation}
\begin{figure}[H]
    \centering
    \begin{tikzpicture}
        \coordinate (O) at (0,0);  
        \coordinate (rx) at ({2*cos(45)},{2*sin(45)});  
        \coordinate (Rx) at ({4*cos(45)},{4*sin(45)});  
        \coordinate (x) at ({5.3*cos(45)},{5.3*sin(45)});   
    
        \draw[dotted, thick] (rx) -- (x);
    
        \draw[thick] (Rx) arc[start angle=45, end angle=90, radius=4];
        \draw[thick] (Rx) arc[start angle=45, end angle=0, radius=4];
    
        \fill (rx) circle (1.5pt) node[below right] {$\frac{rx}{|x|}$};
        \fill (Rx) circle (1.5pt) node[right] {$\frac{Rx}{|x|}$};
        \fill (x) circle (1.5pt) node[below right] {$x$};
    \end{tikzpicture}
    \caption{For $|x|>R$, $\frac{Rx}{|x|}$ is a convex combination of $\frac{rx}{|x|}$ and $x$.}
    \label{ConvexCombFig}
\end{figure}
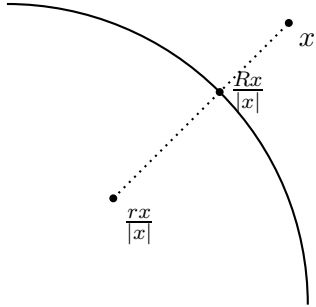

\smallskip
\noindent
\textit{Step 6:} Upper bound for $|x|\geq r$. For $|x|\geq r$ we can write $x= \frac{r}{R+1} \frac{(R+1)x}{r}$ and use the convexity of $\phi$ to obtain
\begin{equation}
\label{AppendixEq4}
    \phi(x)\leq \frac{r}{R+1}\phi\left(\frac{(R+1)x}{r}\right) + \left( 1-\frac{r}{R+1}\right)\phi(0). 
\end{equation}
Using \eqref{AppendixEq1} for $x=0$, $\phi(0)\leq m_0$. In addition, for $|x|\geq r$, we have $|\frac{(R+1)x}{r}|\geq R+1> R$, thus \eqref{AppendixEq3} applies to give an upper bound for $\phi(\frac{(R+1)x}{r})$. Consequently, \eqref{AppendixEq4} becomes
\begin{equation}
\label{AppendixEq5}
\begin{aligned}
    \phi(x) &\leq \frac{|x|}{R-r}+ m_0 - \frac{r^2}{(R+1)(R-r)}, \quad \text{ for } |x|\geq r. 
\end{aligned}
\end{equation}

\smallskip
\noindent
\textit{Step 7:} Choosing the constants $a$ and $b$. Note that $m_0 - \frac{r^2}{(R+1)(R-r)}\leq m_0$. Therefore, setting
\begin{equation*}
    a\defeq \frac{1}{R-r} \quad \text{ and } \quad b\defeq m_0, 
\end{equation*}
we conclude, by \eqref{AppendixEq1} and \eqref{AppendixEq5}, 
$\phi(x)\leq a|x|+b$ for all $x\in \R^n$. 
\end{proof}

\begin{claim}
\label{FiniteVolClaim}
    A convex set with a non-empty interior and finite volume must be bounded. 
\end{claim}
\begin{proof}
    Let us start by making some simplifying assumptions that do not hurt generality. Since $C$ has a non-empty interior, there exists $r>0$ such that a ball of radius $r$ is contained in $C$. By translating $C$, we can assume that $rB_2^n\subset C$. Fix $x_0\in C$. After rotation, we may assume $x_0= (0, \ldots, 0, h)$ for some $h>0$. 

    The goal is to find an upper bound on $|x_0|=h$. Since $x\in C$ and $(rB_2^{n-1})\times\{0\}\subset rB_2^n \subset C$, the cone 
    \begin{equation*}
        C_0\defeq \mathrm{conv}\{x_0, (rB_2^{n-1})\times\{0\}\} = \{((1-\lambda)\xi, \lambda h)\in \R^{n-1}\times \R: \lambda\in [0,1], \xi\in rB_2^{n-1}\},  
    \end{equation*}
    is contained in $C$ because $C$ is convex.
    Using Tonelli's theorem \cite[\S 2.37]{Folland99}, 
    \begin{equation*}
        |C_0|= \int_{0}^h \left| \frac{h-x_n}{h} rB_2^{n-1}\right| dx_n = \frac{r^{n-1}}{h^{n-1}}|B_2^{n-1}| \int_0^h (h-x_n)^{n-1}dx_n. 
    \end{equation*}
    Evaluating the integral, we find $|C_0|= \frac{|B_2^{n-1}|}{n} r^{n-1} h= \frac{|B_2^{n-1}|}{n} r^{n-1} |x_0|$, because $|x_0|=h$. By convexity of $C$, $C_0\subset C$, hence $|C_0|\leq |C|$ and 
    \begin{equation*}
        |x_0|\leq \frac{n|C|}{r^{n-1}|B_2^{n-1}|},
    \end{equation*}
    which proves $C$ is bounded. 
\end{proof}

\begin{corollary}
\label{FinitenessofMoments}
    Let $\phi:\R^n\to \R\cup\{\infty\}$ be a convex function with $0<V(\phi)<\infty$. For $t>0$, 
    \begin{equation*}
        \int_{\R^n}|x|^t e^{-\phi(x)} <\infty. 
    \end{equation*}
\end{corollary}
\begin{proof}
    By Lemma \ref{ConvexLowerBoundLemma}, there exist $a>0$ and $b\in \R$ such that $\phi(x)\geq a|x|+ b$ for all $x\in \R^n$. Therefore, using polar coordinates, 
    \begin{equation*}
        \int_{\R^n}|x|^t e^{-\phi(x)}dx \leq e^{-b} \int_{\R^n}|x|^t e^{-a|x|}dx= |\partial B_2^{n}| e^{-b} \int_0^{\infty}r^{t+n-1}e^{-ar}dr
    \end{equation*}
    The integral on the right-hand side can be recognized as $a^{-(t+n)} \Gamma(t+n)$, hence is finite. 
\end{proof}

\phantomsection
\addcontentsline{toc}{section}{References}
\printbibliography

{\sc University of Maryland}

{\tt vmastr@umd.edu}
\end{document}

%% file: macros.tex
\usepackage{amsmath, amsthm, amssymb, mathrsfs, mathtools}
\usepackage{commath}  
\usepackage{esvect}  
\usepackage{bm}  

\usepackage{graphicx, wrapfig, float}
\usepackage{tikz-cd}  
\usetikzlibrary{backgrounds, matrix, fit, decorations.pathreplacing, calc, positioning}

\usepackage{enumitem}  
\usepackage{tocloft}  
\usepackage{caption, subcaption}  
\usepackage{hyperref}  
\usepackage[nottoc]{tocbibind}  
\usepackage{lipsum}  

\usepackage{hyperref}
\usepackage[backend=biber, 
            maxbibnames=99, 
            citestyle=numeric, 
            bibstyle=numeric, 
            sorting=nyt, 
            backref=false]{biblatex}

\renewbibmacro{in:}{} 

\DeclareFieldFormat{journaltitle}{#1\isdot} 

\DeclareFieldFormat*{title}{\mkbibitalic{#1}} 

\DefineBibliographyExtras{english}{%
  }

\DeclareFieldFormat[article]{volume}{\mkbibbold{#1}} 

\usepackage{csquotes}

\hypersetup{
    breaklinks=true,
    colorlinks=true,
    citecolor=blue,
    linkcolor=blue, 
    urlcolor=cyan
}
\PassOptionsToPackage{hyphens}{url}

\newtheorem{theorem}{Theorem}[section]
\newtheorem{lemma}[theorem]{Lemma}
\newtheorem{proposition}[theorem]{Proposition}
\newtheorem{question}[theorem]{Question}

\newtheorem{corollary}[theorem]{Corollary}
\newtheorem{definition}[theorem]{Definition}
\newtheorem{claim}[theorem]{Claim}
\newtheorem{conjecture}[theorem]{Conjecture}

\theoremstyle{definition}
\newtheorem{example}[theorem]{Example}
\newtheorem{remark}[theorem]{Remark}

\newcommand{\R}{\mathbb{R}}

\newcommand{\e}{\varepsilon}

\newcommand{\beq}{\begin{equation}}
\newcommand{\eeq}{\end{equation}}
\newcommand{\bdef}{\begin{definition}}
\newcommand{\eedef}{\end{definition}}
\newcommand{\bpf}{\begin{proof}}
\newcommand{\epf}{\end{proof}}
\newcommand{\bconj}{\begin{conjecture}}
\newcommand{\econj}{\end{conjecture}}
\newcommand{\blem}{\begin{lemma}}
\newcommand{\elem}{\end{lemma}}
\newcommand{\bexample}{\begin{example}}
\newcommand{\eexample}{\end{example}}
\newcommand{\defeq}{\vcentcolon=}

\newcommand{\Cvx}{\mathrm{Cvx}}
\newcommand{\M}{\mathcal{M}}

\makeatletter
\newcommand\setItemnumber[1]{\setcounter{enum\romannumeral\@enumdepth}{\numexpr#1-1\relax}}
\makeatother

\def\frame#1#2#3#4{\hbox{\vbox{\hrule height#1pt
 \hbox{\vrule width#1pt\kern #2pt
 \vbox{\kern #2pt
 \vbox{\hsize #3\noindent #4}
 \kern#2pt}
 \kern#2pt\vrule width #1pt}
 \hrule height0pt depth#1pt}}}
